\documentclass[preprint,12pt]{article}

\RequirePackage[colorlinks,citecolor=blue,urlcolor=blue]{hyperref}
\usepackage{amssymb}
\usepackage{mathrsfs}
\usepackage{graphicx}
\usepackage{colortbl,dcolumn}
\usepackage{tabularx}
\usepackage{amsmath}
\usepackage{psfrag}
\usepackage{booktabs}
\usepackage{cite}
\usepackage{amsthm}

\numberwithin{equation}{section}
\usepackage[top=1.2in, bottom=1.2in, left=1in, right=1in]{geometry}

\newcommand{\R}{\mathbb{R}}
\newcommand{\N}{\mathbb{N}}

\newcommand{\bi}{\mathbf{i}}

\newtheorem{tm}{Theorem}[section]
\newtheorem{df}{Definition}
\newtheorem{lm}{Lemma}
\newtheorem{prop}{Proposition}[section]
\newtheorem{cor}{Corollary}
\newtheorem{rk}{Remark}

\newtheorem{ap}{Assumption}

\allowdisplaybreaks \allowdisplaybreaks[4]%change lines

\begin{document} 

\title{Approximation of Invariant Measure for Damped Stochastic Nonlinear Schr\"{o}dinger Equation via an Ergodic Numerical Scheme\footnotemark[1]}
       \author{
        Chuchu Chen, Jialin Hong, and Xu Wang\footnotemark[2]\\
       {\small  Institute of Computational Mathematics and Scientific/Engineering Computing,}\\{\small Academy of Mathematics and Systems Science, Chinese Academy of Sciences, }\\
         {\small Beijing 100190, P.R.China }}
       \maketitle
       \footnotetext{\footnotemark[1]Authors are supported by National Natural Science Foundation of China (NO. 91530118, NO. 91130003, NO. 11021101 and NO. 11290142).}
        \footnotetext{\footnotemark[2]Corresponding author: wangxu@lsec.cc.ac.cn}

       \begin{abstract}
         {\rm\small In order to inherit numerically the ergodicity of the damped stochastic nonlinear Schr\"odinger equation with additive noise, we propose a fully discrete scheme, whose spatial direction is based on spectral Galerkin method and temporal direction is based on a modification of the implicit Euler scheme. We not only prove the unique ergodicity of the numerical solutions of both spatial semi-discretization and full discretization, but also present error estimations on invariant measures, which gives order $2$ in spatial direction and order ${\frac12}$ in temporal direction.}\\


\end{abstract}

\section{\textsc{\Large{I}ntroduction}}

The ergodicity of stochastic differential equations (SDEs) and stochastic partial differential equations (SPDEs) characterizes the longtime behavior of the solutions (see \cite{debussche,mattingly,daprato} and references therein), and it is natural to construct proper numerical schemes which could inherit the ergodicity.
For ergodic SDEs with bounded or global Lipschitz coefficients,  the ergodicity of several schemes were studied in \cite{talay}. It also gave an error estimation of invariant measures
\begin{equation*}
e(\phi)=\left|\int\phi(y)d\mu(y)-\int\phi(y)d\tilde{\mu}(y)\right|
\end{equation*}
via the exponential decay property of the solution of Kolmogorov equation, where $\mu$ and $\tilde{\mu}$ denote the original invariant measure and the numerical one respectively. In the local Lipschitz case, the ergodicity is inherited by specially constructed implicit discretizations (see \cite{mattingly} and references therein).
For SDEs, there are also various works related to the study of error $e(\phi)$ by assuming the ergodicity of the schemes (see \cite{abdulle} and references therein).
For SPDEs, there have also been some significant results concentrating on invariant laws, e.g., \cite{brehier} studied a semi-implicit Euler scheme in temporal direction with respect to parabolic type SPDEs with bounded nonlinearity and space-time white noise; \cite{brehier2} studied a full discretization for stochastic evolution equations with global Lipschitz nonlinearity and space-time white noise. Invariant laws of the approximations are, in general, possibly not unique.
 To our knowledge, there has been less work on constructing a fully discrete scheme to inherit the unique ergodicity of SPDEs up to now.

In this paper, we consider an initial-boundary problem of an ergodic one-dimensional damped stochastic  nonlinear Schr\"odinger equation
\begin{equation}\label{model}
\left\{
\begin{aligned}
&du=\big{(}\mathbf{i}\Delta u-\alpha u+\mathbf{i}\lambda|u|^2u\big{)}dt+Q^{\frac{1}{2}}dW\\
&u(t,0)=u(t,1)=0,~t\geq0\\
&u(0,x)=u_0(x),~x\in[0,1],
\end{aligned}
\right.
\end{equation}
where $\alpha>0,$ $\lambda=\pm1$ and the solution $u$ is a complex valued ($\mathbb{C}$-valued) random field on a probability space $(\Omega,\mathcal{F},P)$. 
The noise term involves a cylindrical Wiener process $W$ and a symmetric, positive, trace class operator $Q$ such that the noise is colored in space and white in time. The operator $Q$ is supposed to commute with Laplacian $\Delta$, and the noise has the following Karhunen-Loeve expansion
$$Q^{\frac{1}{2}}dW=\sum_{m=1}^{\infty}\sqrt{\eta_m}e_m(x)d\beta_m(t),\;\;\eta_m\in\mathbb{R^+}\;\;\text{and}\;\;\eta:=\sum_{m=1}^{\infty}\eta_m<\infty,$$
where $\{\beta_m(t)\}_{m\geq1},$ associated to a filtration $\{\mathcal{F}_t\}_{t\geq 0},$ is a family of independent and identically distributed  $\mathbb{C}$-valued Wiener processes and $\{e_m\}_{m\ge1}$ is the eigenbasis of the Dirichlet Laplacian.
This model has many applications in statistical physics and has been studied by many authors. For instance, it can describe the transmission of the signal along the fiber line with signal loss (see \cite{FK01,FK04} and references therein).
The ergodicity for (\ref{model}) with $\lambda=1$ has been studied in \cite{debussche} based on a coupling method,  Foias-Prodi type estimates and a priori estimates for a modified Hamiltonian $\mathcal{H}=\frac12\|\cdot\|_1^2-\frac14\|\cdot\|_{L^4}^4+c_0\|\cdot\|_0^6$. 
The authors showed that \eqref{model} possesses a unique invariant measure $\mu$ assuming that the noise is non-degenerate in the low modes, i.e., $\eta_m>0$, $m\le N_*$ for some sufficiently large $N_*$.
In the same procedure, one can also show the ergodicity for the cases $\lambda=0$ and $\lambda=-1$ by setting 
$\mathcal{H}=\frac12\|\cdot\|_1^2-\frac{\lambda}4\|\cdot\|_{L^4}^4+c_0\|\cdot\|_0^6$.
 Note that the damped term $(\alpha>0)$ is necessary for both linear and nonlinear Schr\"{o}dinger equation to be ergodic. 

Our work mainly focuses on the construction of a fully discrete and uniquely ergodic numerical scheme (i.e., whose numerical solution possesses a unique invariant measure). Moreover, 
the estimation of error between the original invariant measure and the numerical one is also considered based on the weak error of solutions.

In order to obtain a scheme whose noise remains in an explicit expression, we apply spectral Galerkin method
 in spatial direction 
 to obtain a $N$-dimensional SDE 
  \begin{equation}\label{1.2}
 du_N=\Big{(}\mathbf{i}\Delta u_N-\alpha u_N+\mathbf{i}\lambda\pi_N\left(|u_N|^2u_N\right)\Big{)}dt+\pi_NQ^{\frac{1}{2}}dW
 \end{equation}
 with $\pi_N$ being a projection operator. Here the spectral Galerkin method also ensures that the semigroup operator is the same as the one of \eqref{model}, which simplifies the error estimate in spatial direction.
 We find a Lyapunov function by proving the uniform boundedness of $u_N$ in $L^2$-norm. It ensures the existence of the invariant measure of (\ref{1.2}).
We show that the solution $u_N(t)$ is a strong Feller and irreducible process via the non-degeneracy of the noise term in (\ref{1.2}). 
Hence, $u_N(t)$ possesses a unique invariant measure $\mu_N$, which implies the ergodicity of $u_N(t)$.
We would like to emphasize that the noise in the original equation do not need to be non-degenerate. Our method is also available under the same assumption in \cite{debussche}, that is $\eta_m>0$, $m<N_*$ for some sufficiently large $N_*$. Here $N$ and $N_*$ need to satisfy the condition $N<N_*$ to ensure the non-degeneracy for the truncated noise and obtain the ergodicity for numerical solutions.
The error between invariant measures $\mu_N$ and $\mu$ is transferred into the weak error of the solutions, which is required to be independent of time $t$. 
Different from conservative equations, the damped term in (\ref{model}) and (\ref{1.2}) contributes to an exponential estimate on the difference between semigroup operators $S(t)$ and $S(t)\pi_N$, where $S(t)$ is generated by the linear operator $\mathbf{i}\Delta-\alpha$. Therefore, 
we achieve the time-independent weak error of solutions directly which, together with the ergodicity of $u$ and $u_N$, deduces the error between invariant measures $\mu_N$ and $\mu$.
%avoiding proving the exponential decay of the solution of Kolmogorov equation. 

For the temporal discretization of \eqref{1.2}, we propose a new scheme
\begin{equation}\label{1.3}
u_N^{k}-e^{-\alpha\tau}u_N^{k-1}=\left(\mathbf{i}\Delta u_N^{k}+\mathbf{i}\lambda\pi_N\left( \frac{|u_N^{k}|^2+|e^{-\alpha\tau}u_N^{k-1}|^2}{2}u_N^{k}\right)\right)\tau+\pi_NQ^{\frac{1}{2}}\delta W_{k},
\end{equation}
which is a modification of the implicit Euler scheme.
In order to analyze the effect of the time discretization, we investigate both the ergodicity of $u_N^k$ and the weak error between $u_N$ and $u_N^k$. 
The fully discrete scheme (\ref{1.3}) is specially constructed to ensure the uniform boundedness of $u_N^k$ in  $L^2$-, $\dot{H}^1$- and $\dot{H}^2$-norms, which is essential to obtain the existance of the invariant measure as well as the time-independence of the weak error.
Together with the Brouwer fixed point theorem and properties of homogeneous Markov chains, we prove that $u_N^k$ is uniquely ergodic. 
For the weak error, it is usually analyzed in a finite time interval $[0,T]$ and depends on $T$ (see e.g. \cite{BD06,debussche2}). 
In our cases, however,  the weak error between $u_N(T)$ and $u_N^M(T)$ is required to be independent of time $T$ and step $M$. 
 %Thus, we prove the exponential decay of the difference between operators $S(t)$ and $S_{\tau}$, which is a crucial condition for the solutions to have a time-independent weak error. Moreover, an exponential estimate of the difference between the unbounded and nonlinear terms in spatial semi-discretization and full discretization is also required. 
Thus, some technical estimates are given to obtain the exponential decay of the difference between non-global Lipschitz nonlinear terms and between $S(t)$ and $S_{\tau}$. Based on the time-independency of the weak error of the solutions, we show that the error of invariant measures has at least the same order as the weak error of the solutions.

This paper is organized as follows. In section 2, some notations and definitions about ergodicity are introduced. In section 3, we apply spectral Galerkin method to (\ref{model}) and prove the ergodicity of the spatial semi-discrete scheme.  The time-independent weak error of the solutions, together with the error between invariant measures, is given. Section 4 is devoted to the proof of ergodicity of the fully discrete scheme. Moreover, we give the approximation error of invariant measure in temporal direction via the time-independent weak error. 
The last section is the appendix of some proofs.

\vspace{5mm}

%%%%%%%%%%%%%%%%%%%%%%%
\section{\textsc{\Large{P}reliminaries}}
In this section, we present some notations and the definition of ergodicity. Moreover, we introduce a sufficient condition for a stochastic process to be ergodic, which will be used in our proof on ergodicity of the numerical solution. 
%%%%%%%%%%%
\subsection{Notations}

We set 
the linear operator $A:=-\mathbf{i}\Delta+\alpha,$
and the semigroup $S(t):=e^{-tA}=e^{t(\mathbf{i}\Delta-\alpha)}$ is generated by $A$. 
The mild solution of (\ref{model}) exists globally and can be written as
$$u(t)=S(t)u_0+\mathbf{i}\lambda\int_0^tS(t-s)|u(s)|^2u(s)ds+\int_0^tS(t-s)Q^{\frac{1}{2}}dW(s).$$
It is obvious that $\{\lambda_n\}_{n\in\mathbb{N}}:=\big{\{}\mathbf{i}(n\pi)^2+\alpha\big{\}}_{n\in\mathbb{N}}$ is a sequence of eigenvalues of $A$ with $1\leq|\lambda_n|\rightarrow+\infty$ and
$\{e_n\}_{n\in\mathbb{N}}:=\big{\{}\sqrt{2}\sin{n\pi x}\big{\}}_{n\in\mathbb{N}}$ is the associated eigenbasis of $A$ with Dirichlet boundary condition. Denoting $L_0^2(0,1)$ as the space $L^2(0,1)$ with homogenous Dirichlet boundary condition, then $\{e_n\}_{n\in\mathbb{N}}$ is an orthonormal basis of $L_0^2(0,1)$.

\begin{df}
For all $s\in\mathbb{N}$, we define the normed linear space
$$\dot{H}^s:=D(A^{\frac{s}2})=\Big{\{}u\Big{|}u=\sum_{n=1}^{\infty}(u,e_n)e_n\in L_0^2(0,1)~s.t.~\sum_{n=1}^{\infty}\big{|}(u,e_n)\big{|}^2|\lambda_n|^s<\infty\Big{\}},$$
endowed with the $s$-norm
$$\|u\|_s:=\left(\sum_{n=1}^{\infty}\big{|}\left(u,e_n\right)\big{|}^2\left|\lambda_n\right|^s\right)^{\frac{1}{2}},$$
where the inner product in the complex Hilbert space $L^2(0,1)$ is defined by
$$(u,v)=\int_0^1u(x)\overline{v}(x)dx,~\forall\,u,v\in L^2(0,1).$$
In particular,  $\|u\|_0=\|u\|_{L^2}, \forall\,u\in \dot{H}^0$.
\end{df}
In the sequel, we use notations $L^2:=L^2(0,1)$ and $H^s:=H^s(0,1)$. It's easy to check that the above norms satisfy $\|u\|_r\leq\|u\|_s (\forall\,0\leq r\leq s)$ and $\|u\|_s\cong\|u\|_{H^s} (s=0,1,2)$ for any  $u\in\dot{H}^s$.

The operator norm is defined as
$$\|B\|_{\mathcal{L}(\dot{H}^s,\dot{H}^r)}=\sup_{u\in \dot{H}^s}\frac{\|Bu\|_r}{\|u\|_s},\;\forall\,r, s\in\mathbb{N},$$
hence, for $0\leq r\leq s$,
$$\|S(t)\|_{\mathcal{L}(\dot{H}^s,\dot{H}^r)}=\sup_{u\in \dot{H}^s}\frac{\left(\sum_{n=1}^{\infty}\big{|}\left(e^{t(\mathbf{i}\Delta-\alpha)}u,e_n\right)\big{|}^2\left|\lambda_n\right|^r\right)^{\frac{1}{2}}}{\|u\|_s}=\sup_{u\in \dot{H}^s}\frac{e^{-\alpha t}\|u\|_r}{\|u\|_s}\leq e^{-\alpha t}.$$
We need $Q^{\frac12}$ to be a Hilbert-Schmidt operator from $L^2$ to $\dot{H}^s$ with norm
$$\|Q^{\frac12}\|^2_{\mathcal{HS}(L^2,\dot{H}^s)}
:=\sum_{m=1}^{\infty}\|Q^{\frac12}e_m\|_{s}^2
=\sum_{m=1}^{\infty}|\lambda_m|^s\eta_m<\infty.$$
%Here $\mathcal{L}_2^s:=\mathcal{HS}(L^2,\dot{H}^s)$ denotes the space of Hilbert-Schmidt operators from $L^2$ to $\dot{H}^s$.
Assumptions on $s$ will be given below.
%%%%%%%%
\subsection{Ergodicity}
Let $P_t$ be the Markov transition semigroup with an invariant measure $\mu$ and $V$ be a Hilbert space. 
The Von Neumann theorem ensures that the limit
$$\lim\limits_{T\rightarrow\infty}\frac{1}{T}\int_0^TP_t\phi(y)dt, \;\;\phi\in L^2(V,\mu)$$ always exists in $L^2(V,\mu)$, where $y$ denotes the initial value of the stochastic process.
\begin{df}(see e.g. \cite{daprato})
If $P_t$ has an invariant measure $\mu$, and in addition it happens that
\begin{equation}\label{df}
\lim\limits_{T\rightarrow\infty}\frac{1}{T}\int_0^TP_t\phi(y)dt=\int_V\phi d\mu\;\;\; in \;\;L^2(V,\mu)
\end{equation}
for all $\phi\in L^2(V,\mu)$. Then $P_t$ is said to be ergodic.
\end{df}
\begin{rk}
In the following sections, we choose $P_t\phi(u_0)=E[\phi(u(t))|u(0)=u_0]$ for any deterministic initial value $u_0$, and take expectation of both sides of (\ref{df}) to obtain
\begin{equation}\label{df2}
\lim\limits_{T\rightarrow\infty}\frac{1}{T}\int_0^TE[\phi(u)]dt=\int_V\phi d\mu\;\;\; in \;\;\mathbb{R}.
\end{equation}
\end{rk}
The sufficient conditions for a stochastic process to be ergodic are stated in the following theorem.
\begin{tm}(see e.g. \cite{daprato})\label{ergodicity}
Let $F:V\rightarrow[0,\infty]$ be a Borel function (Lyapunov function) whose level sets $$K_a:=\{x\in V:F(x)\leq a\}$$
are compact for any $a>0$. Assume that there exists $y \in V$ and  $C(y)>0$ such that
$$E\Big{[}F\big{(}u(t;y)\big{)}\Big{]}\leq C(y)\;\;\;for\;\;all\;\;t\in\mathbb{R^+},$$
where $u(t;y)$ denotes a stochastic process whose start point is y. Then u has at least one invariant measure.

If in addition the associated semigroup $P_t$ is strong Feller and irreducible, then u possesses a unique invariant measure. Thus, u is ergodic.
\end{tm}

For \eqref{model}, it is ergodic with a unique invariant measure. 
\begin{tm}(see \cite{debussche})
There exists a unique stationary probability measure $\mu$ of $\{P_t\}_{t\in\mathbb{R^+}}$ on $H_0^1(0,1)$. Moreover, for any $p\in\mathbb{N}\verb{\{ \{0\}$, $\mu$ satisfies
$$\int_{H_0^1(0,1)}\|u\|_1^{2p}d\mu<\infty.$$
\end{tm}

\vspace{5mm}

%%%%%%%%%%%%%%%%%%%%%%%%%
\section{\textsc{\Large{S}patial semi-discretization}}
We apply spectral Galerkin method to problem (\ref{model}) to get a spatial semi-discrete scheme which is a finite-dimensional SDE. We show that the solution $u_N$ of (\ref{spectral}) possesses a unique invariant measure $\mu_N$, which leads to the ergodicity of $u_N$. Furthermore, we prove that the weak error of the spatial semi-discrete scheme does not depend on the time interval, which implies that $\mu_N$ converges to $\mu$ in at least the same rate.

%%%%%%%%%%%%%%%%%%%%
\subsection{Spectral Galerkin method}
The finite-dimensional spectral space is defined as  $$V_N:=span\{e_m\}_{m=1}^N.$$
Let
$\pi_N:\dot{H}^0\rightarrow V_N$ be a projection operator, which is defined as
$$\pi_Nu=\sum_{m=1}^N(u,e_m)e_m,~\forall\,u=\sum_{m=1}^{\infty}(u,e_m)e_m\in \dot{H}^0.$$
We use $u_N$ as an approximation to the original solution $u$, and the spatial semi-discrete scheme is expressed as
\begin{equation}\label{spectral}
\left\{
\begin{aligned}
&du_N=\Big{(}\mathbf{i}\Delta u_N-\alpha u_N+\mathbf{i}\lambda\pi_N\left(|u_N|^2u_N\right)\Big{)}dt+\pi_NQ^{\frac{1}{2}}dW\\
&u_N(0,x)=\pi_Nu_0(x),
\end{aligned}
\right.
\end{equation}
where $\pi_NQ^{\frac{1}{2}}dW=\sum_{m=1}^N\sqrt{\eta_m}e_m(x)d\beta_m(t)$, 
and the projection operator $\pi_N$ is bounded
$$\|\pi_N\|_{\mathcal{L}(\dot{H}^s,L^2)}\leq 1,\;\;\forall\,s\in\mathbb{N}.$$

%%%%%%%%%%%%%%%%%%%%%%   
\subsection{Ergodicity of spatial semi-discrete scheme}\label{Ergodicity of space semi-discrete scheme}
\begin{tm}\label{spaceergodic}
Let $u_N(t,x)$ be the solution of equation (\ref{spectral}), then $u_N$ possesses a unique invariant measure, denoted by $\mu_N$. Thus, $u_N$ is ergodic.
\end{tm}

\begin{proof}
Following from Theorem \ref{ergodicity}, we need to show three properties of $u_N$, "strong Feller", "irreducibility" and "Lyapunov condition", in order to show the ergodicity of $u_N$. Thus the proof is divided into three parts as follows.

%%%%%%%%%%%%%%%
\textbf{Part 1. Strong Feller.} 
We transform (\ref{spectral}) into its equivalent finite-dimensional SDE form. Denote $a_m(t)=\big{(}u_N(t,x),e_m(x)\big{)}$ and we have $$u_N(t,x)=\sum_{m=1}^Na_m(t)e_m(x).$$
Applying the It\^{o}'s formula to $a_m(t)$ leads to
$$da_m(t)=\Big{[}-\lambda_ma_m(t)+\left(\mathbf{i}\lambda\pi_N\left(|u_N|^2u_N\right),e_m\right)\Big{]}dt+\sqrt{\eta_m}d\beta_m(t),\quad1\le m\le N.$$
We decompose the above equation into its real and imaginary parts by 
denoting $a_m=a_{m}^1+\mathbf{i}a_{m}^2,~\lambda_m=\lambda_{m}^1+\mathbf{i}\lambda_{m}^2$ and $\beta_m=\beta_{m}^1+\mathbf{i}\beta_{m}^2$, where $\{\beta_{m}^i\}_{1\le m\le N,i=1,2}$ is a family of  independent $\mathbb{R}$-valued Wiener processes and the superscripts $1$ and $2$ mean the real and imaginary parts of a complex value, respectively, and obtain
\begin{equation*}
\left\{
\begin{aligned}
da_m^1=\Big{[}-\lambda_m^1a_m^1+\lambda_m^2a_m^2+Re\left(\mathbf{i}\lambda\pi_N\left(|u_N|^2u_N\right),e_m\right)\Big{]}dt+\sqrt{\eta_m}d\beta_m^1(t),\\
da_m^2=\Big{[}-\lambda_m^2a_m^1-\lambda_m^1a_m^2+Im\left(\mathbf{i}\lambda\pi_N\left(|u_N|^2u_N\right),e_m\right)\Big{]}dt+\sqrt{\eta_m}d\beta_m^2(t).
\end{aligned}
\right.
\end{equation*}
 With notations
\begin{equation*}X(t)=
\left(
\begin{array}{c}
a_{1}^1(t)\\
a_{1}^2(t)\\
\vdots\\
a_{N}^1(t)\\
a_{N}^2(t)
\end{array}
\right)\in \mathbb{R}^{2N},\;\;F=
\left(
\begin{array}{ccc}
\Lambda_1& & \\
 &\ddots& \\
 & & \Lambda_N
\end{array}
\right),\;\;\Lambda_i=
\left(
\begin{array}{cc}
-\lambda_{i}^1&\lambda_{i}^2\\
-\lambda_{i}^2&-\lambda_{i}^1
\end{array}
\right),
\end{equation*}

\begin{equation*}G(X(t))=
\left(
\begin{array}{c}
Re~\left(\mathbf{i}\lambda\pi_N\left(|u_N|^2u_N\right),e_1\right)\\
Im~\left(\mathbf{i}\lambda\pi_N\left(|u_N|^2u_N\right),e_1\right)\\
\vdots\\
Re~\left(\mathbf{i}\lambda\pi_N\left(|u_N|^2u_N\right),e_N\right)\\
Im~\left(\mathbf{i}\lambda\pi_N\left(|u_N|^2u_N\right),e_N\right)
\end{array}
\right),\;\;\beta=
\left(
\begin{array}{c}
\beta_{1}^1\\
\beta_{1}^2\\
\vdots\\
\beta_{N}^1\\
\beta_{N}^2
\end{array}
\right),
\end{equation*}
and
\begin{equation*}Z=
\left(
\begin{array}{ccccc}
\sqrt{\eta_1}&&&&\\
&\sqrt{\eta_1}&&&\\
&&\ddots&&\\
&&&\sqrt{\eta_N}&\\
&&&&\sqrt{\eta_N}
\end{array}
\right):=(Z_{1}^1,Z_{1}^2\cdots,Z_{N}^1,Z_{N}^2),
\end{equation*}
we get an equivalent form of (\ref{spectral})

\begin{align*}
dX(t)=&\Big{[}FX(t)+G\big{(}X(t)\big{)}\Big{]}dt+\sum_{m=1}^N\sum_{i=1}^2Z_{m}^id\beta_{m}^i,\nonumber\\
:=&Y\left(X(t)\right)+\sum_{m=1}^N\sum_{i=1}^2Z_{m}^id\beta_{m}^i.
\end{align*}
It is obvious that
$$span\{Z_{1}^1,Z_{1}^2,\cdots,Z_{N}^1,Z_{N}^2\}=\mathbb{R}^{2N},$$
which means the H\"{o}rmander's condition holds. 
According to the H\"{o}rmander theorem\cite{hormander}, $X(t)$ is a strong Feller process.

%%%%%%%%%%%%%%%%%%%%%
\textbf{Part 2. Irreducibility.} 
By using the same notations as above, we have
\begin{equation}\label{A}
dX=Y(X)dt+Zd\beta,
\end{equation}
with $X=X(t)\in\mathbb{R}^{2N},~X(0)=y$ and $Z$ being invertible.
Using a similar technique as \cite{mattingly}, we consider the associated control problem
\begin{equation}\label{control}
d\overline{X}=Y(\overline{X})dt+ZdU,
\end{equation}
with $\overline{X}=\overline{X}(t)$ and a smooth control function $U\in C^1(0,T)$.
For any fixed $T>0$, $y\in \mathbb{R}^{2N}$ and $y^+\in \mathbb{R}^{2N}$, using polynomial interpolation, we derive a continuous function $\left(\overline{X}(t),~{t\in[0,T]}\right)$ such that
$$\overline{X}(0)=y,\;\;\overline{X}(T)=y^+.$$
Hence, $$dU=Z^{-1}\big{(}d\overline{X}-Y(\overline{X})dt\big{)},$$
and we get the control function $U$ such that (\ref{control}) is satisfied with $\overline{X}(0)=y,\,\overline{X}(T)=y^+$ and $U(0)=0$.
We subtract the resulting equations (\ref{A}) and (\ref{control}), and achieve
$$X(t)-\overline{X}(t)=\int_0^tY(X(s))-Y(\overline{X}(s))ds+Z(\beta(t)-U(t)),\;\;t\in[0,T].$$
According to the properties of Brownian motion,
$$P\left(\sup_{0\leq t\leq T}\big{|}\beta(t)-U(t)\big{|}\leq\epsilon\right)>0,\;\;\forall\,\epsilon>0.$$
Note that Y is locally Lipschitz because of its continuous differentiability,  and the ranges of $X(t)$ and $\overline{X}(t)~(t\in[0,T])$ are both compact sets. Thus, it holds 
$$P\left(\big{|}X(t)-\overline{X}(t)\big{|}\leq\int_0^tC_1\big{|}X(s)-\overline{X}(s)\big{|}ds+C_2\epsilon,\;\;\forall\; t\in[0,T]\right)>0,\;\;\forall\;\epsilon>0$$
with $C_1$ and $C_2$ are positive constants independent of $\epsilon$. 
Then the Gr\"onwall's inequality yields
$$P\bigg{(}\big{|}X(t)-\overline{X}(t)\big{|}\leq C_2(1+e^{C_1t})\epsilon,\;\;\forall\;t\in[0,T]\bigg{)}>0,\;\;\forall\;\epsilon>0.$$
For any $\delta>0$, choosing $t=T$ and $\epsilon={\delta}/{C_2(1+e^{C_1T})}>0$, we finally obtain
 $$P\Big{(}|X(T)-y^+|<\delta\Big{)}>0.$$
  In other words, $X(T)$ hits $B(y^+,\delta)$ with positive probability. The irreducibility has been proved.

The above two conditions ensure the uniqueness of the invariant measure of $X(t)$. It suffices to show the existence of invariant measures in the following.

%%%%%%%%%%%%%%
\textbf{Part 3. Lyapunov condition.} 
A useful tool for proving existence of invariant measures is provided by Lyapunov functions, which is introduced in Theorem \ref{ergodicity}.
It\^o's formula applied to $\|u_N(t)\|_0^2$ implies that
\begin{align}\label{un}
d\|u_N(t)\|_0^2=&-2\alpha\|u_N(t)\|_0^2dt+2Re \int_0^1\overline{u}_N(t)\pi_NQ^{\frac{1}{2}}dxdW(t)+2\sum_{m=1}^N\eta_mdt,
\end{align}
where we have used the fact that
\begin{align*}
&Re\left[\mathbf{i}\lambda\int_0^1\pi_N(|u_N|^2u_N)\overline{u}_Ndx\right]\\
=&Re\left[\mathbf{i}\lambda\int_0^1\left(|u_N|^4-(Id-\pi_N)(|u_N|^2u_N)\overline{u}_N\right)dx\right]\\
=&-\lambda Im~\Big{(}(Id-\pi_N)(|u_N|^2u_N),u_N\Big{)}\\
=&0.
\end{align*}
Taking expectation on both sides of \eqref{un}, we get
$$\frac{d}{dt}E\|u_N(t)\|_0^2=-2\alpha E\|u_N(t)\|_0^2+C_N,$$
where $C_N=2\sum_{m=1}^{N}\eta_m\leq2\eta$.
It is solved as
$$E\|u_N(t)\|_0^2=e^{-2\alpha t}\Big{(}\int_0^tC_Ne^{2\alpha s}ds+E\|u_N(0)\|_0^2\Big{)}\leq e^{-2\alpha t}E\|u_N(0)\|_0^2+C,~\forall\,t>0.$$
On the other hand,
\begin{align*}
\|u_N(t)\|_0^2=\int_0^1\Big{|}\sum_{m=1}^{N}a_m(t)e_m(x)\Big{|}^2dx
=\|X(t)\|^2_{l^2(\mathbb{R}^{2N})}.
\end{align*}
Define $F=\|\cdot\|_{l^2(\mathbb{R}^{2N})}:\mathbb{R}^{2N}\rightarrow[0,+\infty]$. The level sets of $F$ are tight by Heine-Borel theorem.
Therefore, $X(t)$ is ergodic.
We mention that the ergodicity of $X(t)$ is equivalent to the existence of a random variable $\xi=(\xi_{1}^1,\xi_{1}^2,\cdots,\xi_{N}^1,\xi_{N}^2)$ such that
\begin{equation*}
\lim_{t\to\infty}X(t)=\xi,~\text{i.e.},~\lim_{t\to\infty}a_{m}^i(t)=\xi_{m}^i,~\forall~m=1,\cdots,N,~i=1,2.
\end{equation*}
It leads to
\begin{equation*}
\lim_{t\to\infty}u_N(t)=\sum_{m=1}^{N}\left(\xi_{m}^1+\mathbf{i}\xi_{m}^2\right)e_m,
\end{equation*}
which shows the ergodicity of $u_N(t)$.
\end{proof}

According to the proof of Lyapunov condition, we have the following uniform boundedness for 0-norm. Moreover, 1-norm is also uniformly bounded, which is also stated in the following proposition. Its proof is given in appendix \ref{5.1} for readers' convenience. In sequel, all the constants $C$ are independent of the end point $T$ of time interval and may be different from line to line.
\begin{prop}\label{1norm}
Assume that $u_0\in \dot{H}^1$, $\|Q^{\frac12}\|_{\mathcal{HS}(L^2,\dot{H}^1)}<\infty$ and $p\ge1$. There exists positive constants  $c_0$ and $C=C(\alpha,p,u_0,c_0,Q)$, such that for any $t>0$,
\begin{align*}
\romannumeral1)&~E\|u_N(t)\|_0^{2p}\leq e^{-2\alpha pt}E\|u_N(0)\|_0^{2p}+C\le C,\\
\romannumeral2)&~E\mathcal{H}(u_N(t))^{p}\leq e^{-\alpha p t}E\mathcal{H}(u_N(0))^{p}+C\le C,
\end{align*}
where $\mathcal{H}(u_N(t))=\frac12\|\nabla u_N(t)\|_0^2-\frac{\lambda}4\|u_N(t)\|_{L^4}^4+c_0\|u_N(t)\|_0^6$.
In addition, if assume further $u_0\in\dot{H}^2$ and $\|Q^{\frac12}\|_{\mathcal{HS}(L^2,\dot{H}^2)}<\infty$, we also have
\begin{align*}
\romannumeral3)&~E\|u_N(t)\|_2^2\leq C.
\end{align*}
\end{prop}
\begin{rk}
The uniform boundedness of the original solution $u$ can also be obtained in the same procedure as Proposition \ref{1norm}. As we require the global well-posedness and high regularity for both the original solution and numerical solutions to obtain the ergodicity as well as the time-independent weak error, the assumptions in this paper (see also \cite{debussche}) are stricter than that in other papers (see e.g. \cite{BD06}).
\end{rk}

%%%%%%   
\subsection{Weak error between solutions $u$ and $u_N$}\label{Weak convergence of space discretization}

Weak convergence is established for the spatial semi-discretization (\ref{spectral}) in this section utilizing a transformation of $u_N(t)$ and the corresponding Kolmogorov equation.
\begin{tm}\label{weakconvergence}
Assume that $u_0\in \dot{H}^2$ and $\|Q^{\frac12}\|_{\mathcal{HS}(L^2,\dot{H}^2)}<\infty$. For any $\phi\in C_b^2(L^2)$, there exists a constant $C=C(u_0,\phi,Q)$ independent of T, such that for any $T>0$, $$\bigg{|}E\Big{[}\phi\big{(}u_N(T)\big{)}\Big{]}-E\Big{[}\phi\big{(}u(T)\big{)}\Big{]}\bigg{|}\leq CN^{-2}.$$
\end{tm}
Before the proof of Theorem \ref{weakconvergence}, we give a useful lemma.
\begin{lm}\label{operator}
Assume that $S(t)$ and $\pi_N$ are defined as before. We have the following estimation
$$\|S(t)-S(t)\pi_N\|_{\mathcal{L}(\dot{H}^s,L^2)}\leq Ce^{-\alpha t}N^{-s}.$$
\end{lm}
\begin{proof}

For any $u\in \dot{H}^s$, we have
\begin{align*}
&\|S(t)u-S(t)\pi_Nu\|_0=e^{-\alpha t}\|u-\pi_Nu\|_0
=e^{-\alpha t}\left(\sum_{n=N+1}^{\infty}|(u,e_n)|^2\right)^{\frac12}\\
&\leq e^{-\alpha t}|\lambda_N|^{-\frac s2}\left(\sum_{n=N+1}^{\infty}|\lambda_n|^s|(u,e_n)|^2\right)^{\frac12}
\leq Ce^{-\alpha t}N^{-s}\|u\|_s.
\end{align*}
\end{proof}
\begin{proof}[Proof of Theorem \ref{weakconvergence}]
\vspace{3mm}
We split the proof in three steps.

%%%%%%%%%%%
\textbf{Step 1.} Calculation of $E\left[\phi(u(T))\right]$.

To eliminate the unbounded Laplacian operator, we consider the modified process
$Y(t)=S(T-t)u(t)$, $t\in[0,T],$
which is the solution of the following SPDE
\begin{align}
dY(t)
&=\mathbf{i}\lambda S(T-t)\Big{[}|S(t-T)Y(t)|^2S(t-T)Y(t)\Big{]}dt+S(T-t)Q^{\frac{1}{2}}dW\nonumber\\
&:=H(Y(t))dt+S(T-t)Q^{\frac{1}{2}}dW.\nonumber
\end{align}
Denote $v(T-t,y):=E[\phi(Y(T))|Y(t)=y]$ and it follows easily
\begin{align*}
\frac{\partial v(T-t,y)}{\partial t}=-\Big{(}Dv(T-t,y),H(y)\Big{)}-\frac{1}{2}Tr\Big{[}(S(T-t)Q^{\frac{1}{2}})^*D^2v(T-t,y)S(T-t)Q^{\frac{1}{2}}\Big{]}.
\end{align*}
Note that the mild solution of $u$ has the expression $u(T)=S(T-t)u(t)+\bi\lambda\int_t^TS(T-s)|u|^2uds+\int_t^TS(T-s)Q^{\frac12}dW$. Thus, we have
\begin{align*}
v(T-t,y)=&E[\phi(Y(T))|Y(t)=y]=E[\phi(u(T))|u(t)=S(t-T)y]\\
=&E\left[\phi\left(y+\bi\lambda\int_t^TS(T-s)|u(s)|^2u(s)ds+\int_t^TS(T-s)Q^{\frac12}dW\right)\right].
\end{align*}
Similarly with \cite{BD06} (Lemma 5.13), for $h\in L^2$, 
\begin{align*}
\left(Dv(T-t,y),h\right)=E\left[\left(D\phi\left(y+\bi\lambda\int_t^TS(T-s)|u(s)|^2u(s)ds+\int_t^TS(T-s)Q^{\frac12}dW\right),\chi^h(t)\right)\right],
\end{align*}
where 
\begin{align*}
\chi^h(t)=&h+\bi\lambda\int_t^TS(T-s)(D(|u(s)|^2u(s)),\chi^h(s))ds\\
=&h+\bi\lambda\int_t^TS(T-s)\left(2|u(s)|^2\chi^h(s)+u^2(s)\overline{\chi^h(s)}\right)ds.
\end{align*}
Based on the uniform boundedness of $\|u\|_1^p$ for $p\ge1$, which can be proved in the same procedure as Proposition \ref{1norm} or \cite{debussche}, the Gr\"onwall's inequality yields $E\|\chi^h(t)\|_0\le C\|h\|_0$. Thus, it holds
\begin{align}\label{dv}
\left|\left(Dv(T-t,y),h\right)\right|\le\|\phi\|_{C_b^1}E\|\chi^h(t)\|_0\le C\|\phi\|_{C_b^1}\|h\|_0.
\end{align}
Similarly, we also have 
\begin{align}\label{d2v}
\left|\Big{(}\left(D^2v(T-t,y),h\right),h\Big{)}\right|\le C\|\phi\|_{C_b^2}\|h\|_0^2.
\end{align}
The It\^o's formula gives that
\begin{align*}
dv(T-t,Y(t))=&\frac{\partial v}{\partial t}(T-t,Y(t))dt
+\left(Dv\left(T-t,Y(t)\right),H\left(Y(t)\right)dt+S(T-t)Q^{\frac{1}{2}}dW(t)\right)\nonumber\\
&+\frac{1}{2}Tr\left[(S(T-t)Q^{\frac{1}{2}})^*D^2v\left(T-t,Y(t)\right)S(T-t)Q^{\frac{1}{2}}\right]dt\nonumber\\
=&\left(Dv(T-t,Y(t)),S(T-t)Q^{\frac{1}{2}}dW(t)\right).
\end{align*}
Therefore,
\begin{equation}\label{eqexact}
v(0,Y(T))=v(T,Y(0))+\int_0^T\left(Dv(T-s,Y(s)),S(T-s)Q^{\frac{1}{2}}dW(s)\right).
\end{equation}
Noticing that $Y(0)=S(T)u_0$ and $Y(T)=u(T)$, we recall $v(T-t,y)=E[\phi(Y(T))|Y(t)=y]$ to derive
\begin{align*}
v(0,Y(T))
=E\left[\phi(u(T))|Y(T)=u(T)\right]
\end{align*}
and
\begin{align*}
v(T,Y(0))=&E\left[\phi(Y(T))|Y(0)=S(T)u_0\right]\nonumber\\
=&E\left[\phi\Big{(}S(T)u_0+\int_0^TH(Y(t))dt+S(T-t)Q^{\frac{1}{2}}dW(t)\Big{)}\Big{|}Y(0)=S(T)u_0\right].\nonumber
\end{align*}
Take expectation of both sides of (\ref{eqexact}) and we have
\begin{equation}\label{exact}
E[\phi(u(T))]=E\left[\phi\Big{(}S(T)u_0+\int_0^TH(Y(t))dt+S(T-t)Q^{\frac{1}{2}}dW(t)\Big{)}\right].\qquad
\end{equation}
%%%%%%%%%%%%%%%

\textbf{Step 2.} Calculation of $E\left[\phi(u_N(T))\right]$.

The mild solution of (\ref{spectral}) is
$$u_N(t)=S(t)\pi_Nu_0+\mathbf{i}\lambda\int_0^tS(t-s)\pi_N\left(|u_N(s)|^2u_N(s)\right)ds+\int_0^tS(t-s)\pi_NQ^{\frac{1}{2}}dW(s).$$
Using similar argument as above, we consider the following stochastic process:
$$Y_N(t)=S(T-t)u_N(t).$$
The relevant SDE is
\begin{align*}
dY_N(t)&=\mathbf{i}\lambda S(T-t)\pi_N\Big{[}|S(t-T)Y_N(t)|^2S(t-T)Y_N(t)\Big{]}dt+S(T-t)\pi_NQ^{\frac{1}{2}}dW\\
&:=H_{N}(Y_N(t))dt+S(T-t)\pi_NQ^{\frac{1}{2}}dW(t).
\end{align*}
Apply It\^{o}'s formula to $t\to v(T-t,Y_N(t))$ and we get
\begin{align*}
dv(T-t,Y_N(t))=&\frac{\partial v}{\partial t}(T-t,Y_N(t))dt\nonumber\\
&+\left(Dv(T-t,Y_N(t)),H_N(Y_N(t))dt+S(T-t)\pi_NQ^{\frac{1}{2}}dW(t)\right)\nonumber\\
&+\frac{1}{2}Tr\Big{[}(S(T-t)\pi_NQ^{\frac{1}{2}})^*D^2v(T-t,Y_N(t))S(T-t)\pi_NQ^{\frac{1}{2}}\Big{]}dt\nonumber\\
=&\left(Dv(T-t,Y_N(t)),S(T-t)\pi_NQ^{\frac{1}{2}}dW(t)\right)\nonumber\\
&+\Big{(}Dv(T-t,Y_N(t)),H_N\left(Y_N(t)\right)-H\left(Y_N(t)\right)\Big{)}dt\nonumber\\
&-\frac{1}{2}Tr\Big{[}(S(T-t)Q^{\frac{1}{2}})^*D^2v(T-t,Y_N(t))S(T-t)Q^{\frac{1}{2}}\Big{]}dt\nonumber\\
&+\frac{1}{2}Tr\Big{[}(S(T-t)\pi_NQ^{\frac{1}{2}})^*D^2v(T-t,Y_N(t))S(T-t)\pi_NQ^{\frac{1}{2}}\Big{]}dt.\nonumber
\end{align*}
Therefore,
\begin{align}\label{eqnum}
v(0,Y_N(T))=&v(T,Y_N(0))+\int_0^T\left(Dv(T-s,Y_N(s)),S(T-s)\pi_NQ^{\frac{1}{2}}dW(s)\right)\nonumber\\
&+\int_0^T\Big{(}Dv\big{(}T-t,Y_N(t)\big{)},H_N\big{(}Y_N(t)\big{)}-H\big{(}Y_N(t)\big{)}\Big{)}dt\nonumber\\
&+\frac{1}{2}\int_0^TTr\left[(S(T-t)\pi_NQ^{\frac{1}{2}})^*D^2v(T-t,Y_N(t))S(T-t)\pi_NQ^{\frac{1}{2}}\right]dt\nonumber\\
&-\frac{1}{2}\int_0^TTr\left[(S(T-t)Q^{\frac{1}{2}})^*D^2v(T-t,Y_N(t))S(T-t)Q^{\frac{1}{2}}\right]dt.
\end{align}
By the construction of $Y_N$, we can check that
$$Y_N(0)=S(T)\pi_Nu_0\;\;\;\text{and}\;\;\;Y_N(T)=u_N(T).$$
According to the representation of $v$, we have
\begin{align*}
v(0,Y_N(T))=E\left[\phi(Y(T))|Y(T)=Y_N(T)\right]
=E\left[\phi(u_N(T))|Y(T)=Y_N(T)\right]
\end{align*}
and
\begin{align*}
v(T,Y_N(0))=&E\left[\phi(Y(T))|Y(0)=S(T)\pi_Nu_0\right]\nonumber\\
=&E\left[\phi\Big{(}S(T)\pi_Nu_0+\int_0^TH(Y(t))dt+S(T-t)Q^{\frac{1}{2}}dW(t)\Big{)}\Big{|}Y(0)=S(T)\pi_Nu_0\right].\;\nonumber
\end{align*}
Take expectation of the two sides of (\ref{eqnum}) and we get
\begin{align}\label{num}
E\left[\phi(u_N(T))\right]=&E\left[\phi\Big{(}S(T)\pi_Nu_0+\int_0^TH(Y(t))dt+S(T-t)Q^{\frac{1}{2}}dW(t)\Big{)}\right]\nonumber\\
&+E\int_0^T\Big{(}Dv\big{(}T-t,Y_N(t)\big{)},H_N\big{(}Y_N(t)\big{)}-H\big{(}Y_N(t)\big{)}\Big{)}dt\nonumber\\
&+\frac{1}{2}E\int_0^T\Bigg{\{}Tr\left[(S(T-t)\pi_NQ^{\frac{1}{2}})^*D^2v(T-t,Y_N(t))S(T-t)\pi_NQ^{\frac{1}{2}}\right]\nonumber\\
&-Tr\left[(S(T-t)Q^{\frac{1}{2}})^*D^2v(T-t,Y_N(t))S(T-t)Q^{\frac{1}{2}}\right]\Bigg{\}}dt.
\end{align}

%%%%%%%%%%%%%

\textbf{Step 3.} Weak error of the solutions.

Subtracting the resulting equations (\ref{exact}) and (\ref{num}) leads to
\begin{align}\label{error}
&E\left[\phi(u_N(T))\right]-E\left[\phi(u(T))\right]\nonumber\\
=&E\bigg{[}\phi\Big{(}S(T)\pi_Nu_0+\int_0^TH(Y(t))dt+S(T-t)Q^{\frac{1}{2}}dW(t)\Big{)}\nonumber\\
&-\phi\Big{(}S(T)u_0+\int_0^TH(Y(t))dt+S(T-t)Q^{\frac{1}{2}}dW(t)\Big{)}\bigg{]}\nonumber\\
&+E\int_0^T\Big{(}Dv\big{(}T-t,Y_N(t)\big{)},H_N\big{(}Y_N(t)\big{)}-H\big{(}Y_N(t)\big{)}\Big{)}dt\nonumber\\
&+\frac{1}{2}E\int_0^T\Bigg{\{}Tr\left[(S(T-t)\pi_NQ^{\frac{1}{2}})^*D^2v(T-t,Y_N(t))S(T-t)\pi_NQ^{\frac{1}{2}}\right]\nonumber\\
&-Tr\left[(S(T-t)Q^{\frac{1}{2}})^*D^2v(T-t,Y_N(t))S(T-t)Q^{\frac{1}{2}}\right]\Bigg{\}}dt\nonumber\\
:=&\uppercase\expandafter{\romannumeral1}+\uppercase\expandafter{\romannumeral2}+\uppercase\expandafter{\romannumeral3}.
\end{align} 
Due to Lemma \ref{operator}, terms $\uppercase\expandafter{\romannumeral1}$ and $\uppercase\expandafter{\romannumeral2}$ can be estimated as
\begin{align}\label{1}
\left|\uppercase\expandafter{\romannumeral1}\right|\leq C\left\|\phi\right\|_{C_b^1}E\left\|S(T)u_0-S(T)\pi_Nu_0\right\|_0
\leq Ce^{-\alpha T}\|\phi\|_{C_b^1}E\|u_0\|_2N^{-2}
\leq Ce^{-\alpha T}N^{-2},
\end{align}
and
\begin{align}\label{2}
|\uppercase\expandafter{\romannumeral2}|\leq&C E\int_0^T\|\phi\|_{C_b^1}\|H_N(Y_N(t))-H(Y_N(t))\|_0dt\nonumber\\
=&C E\int_0^T\|\phi\|_{C_b^1}\|\mathbf{i}\lambda S(T-t)(Id-\pi_N)\big{(}|u_N(t)|^2u_N(t)\big{)}\|_0dt\nonumber\\
\leq&|\lambda|C\int_0^Te^{-\alpha(T-t)}\|\phi\|_{C_b^1}E\Big{[}\|u_N(t)\|_1^2\|u_N(t)\|_2\Big{]}N^{-2}dt\nonumber\\
\leq&|\lambda|\frac C{\alpha}N^{-2}
\end{align}
based on Lemma \ref{operator}, Proposition \ref{1norm}  and the embedding $H^1\hookrightarrow L^{\infty}$ in $\mathbb{R}$.
In the first step of \eqref{2}, we have used the fact \eqref{dv}.

Let us now estimate term $\uppercase\expandafter{\romannumeral3}$.
As $(S(T-t)\pi_N-S(T-t))Q^{\frac{1}{2}}$ is a bounded linear operator and so is $D^2v$ shown in \eqref{d2v}, we have
\begin{align*}
&\bigg{|}Tr\left[(S(T-t)\pi_NQ^{\frac{1}{2}})^*D^2v(T-t,Y_N(t))S(T-t)\pi_NQ^{\frac{1}{2}}\right]\\
&-Tr\left[(S(T-t)Q^{\frac{1}{2}})^*D^2v(T-t,Y_N(t))S(T-t)Q^{\frac{1}{2}}\right]\bigg{|}\\
=&\left|Tr\left[((S(T-t)\pi_N-S(T-t))Q^{\frac{1}{2}})^*D^2v(T-t,Y_N(t))(S(T-t)\pi_N+S(T-t))Q^{\frac{1}{2}}\right]\right|\\
\leq&C\|S(T-t)\pi_N-S(T-t)\|_{\mathcal{L}(\dot{H}^2,L^2)}\|Q^{\frac12}\|_{\mathcal{HS}(L^2,\dot{H}^2)}\|\phi\|_{C_b^2}\|S(T-t)\|_{\mathcal{L}(L^2,L^2)}\|Q^{\frac12}\|_{\mathcal{HS}(L^2,L^2)}\\
\leq& Ce^{-{\alpha}(T-t)}N^{-2}.
\end{align*}
Hence, integrating above equation leads to
\begin{equation}\label{3}
|\uppercase\expandafter{\romannumeral3}|
\leq \frac C{\alpha}N^{-2}.
\end{equation}
Plugging (\ref{1}), (\ref{2}) and (\ref{3}) into (\ref{error}), we get
\begin{equation}\label{ch2}
\bigg{|}E\Big{[}\phi\big{(}u_N(T)\big{)}\Big{]}-E\Big{[}\phi\big{(}u(T)\big{)}\Big{]}\bigg{|}\leq C(e^{-\alpha T}+\frac1{\alpha})N^{-2}\le CN^{-2},
\end{equation}
in which, $C$ is independent of time $T$.
\end{proof}
%%%%%%%%%%%%%%%%%%
\iffalse
\begin{rk}\label{rk2}
For the linear case $\lambda=0$, under additional assumption $u_0\in \dot{H}^2$, we actually get the weak order is twice as high as for the nonlinear case, i.e.,
$$\bigg{|}E\Big{[}\phi\big{(}u_N(T)\big{)}\Big{]}-E\Big{[}\phi\big{(}u(T)\big{)}\Big{]}\bigg{|}\leq CN^{-2}$$
from above proof. Indeed, based on Lemma \ref{operator}, we have 
\begin{align*}
|\uppercase\expandafter{\romannumeral1}|\le C\left\|\phi\right\|_{C_b^1}\left\|S(T)u_0-S(T)\pi_Nu_0\right\|_0
\leq Ce^{-\alpha T}\|\phi\|_{C_b^1}\|u_0\|_2N^{-2}
\leq CN^{-2},
\end{align*}
which yields the result together with \eqref{3} and $\uppercase\expandafter{\romannumeral2}=0$.
\end{rk}
\fi
%%%%%%%%%%%%%%%%%%%%%  
\subsection{Convergence order between invariant measures $\mu$ and $\mu_N$}
By the ergodicity of stochastic processes $u$ and $u_N$, for any deterministic $u_0\in \dot{H}^2$, we have
\begin{align}
\lim_{T\to\infty}\frac{1}{T}\int_0^TE\phi\big{(}u(t)\big{)}dt=\int_{L^2}\phi(y)d\mu(y)
\end{align}
and
\begin{align}
\lim_{T\to\infty}\frac{1}{T}\int_0^TE\phi\big{(}u_N(t)\big{)}dt=\int_{V_N}\phi(y)d\mu_N(y)
\end{align}
for any $\phi\in C_b^2(L^2)$.
Based on the time-independence of the weak error in Theorem \ref{weakconvergence}, it turns out for any fixed $\alpha$ and $N$, 
\begin{align*}
&\left|\int_{L^2}\phi(y)d\mu(y)-\int_{V_N}\phi(y)d\mu_N(y)\right|
=\left|\lim_{T\to\infty}\frac{1}{T}\int_0^TE\phi\big{(}u(t)\big{)}-E\phi\big{(}u_N(t)\big{)}dt\right|\\
\le&\lim_{T\to\infty}\frac{1}{T}\int_0^T\left|E\phi\big{(}u(t)\big{)}-E\phi\big{(}u_N(t)\big{)}\right|dt
\le\lim_{T\to\infty}\frac{1}{T}\int_0^TC(e^{-\alpha t}+\frac1{\alpha})N^{-2}dt\le \frac C{\alpha}N^{-2},
\end{align*}
which implies that $\mu_N$ is a proper approximation of $\mu$. Thus, we give the following theorem.

\iffalse
there exists sufficiently large $T>0,$ 
such that
\begin{align}\label{ergodic}
\left|\frac{1}{T}\int_0^TE\phi\big{(}u(t)\big{)}dt-\int_{L^2}\phi(y)d\mu(y)\right|
\le \frac C{\alpha}N^{-1},\\\label{ergodic2}
\left|\frac{1}{T}\int_0^TE\phi\big{(}u_N(t)\big{)}dt-\int_{V_N}\phi(y)d\mu_N(y)\right|
\le \frac C{\alpha}N^{-1}
\end{align}
and
\begin{align}\label{ergodic3}
\left|\frac{1}{T}\int_0^T\Big{[}E\phi\big{(}u(t)\big{)}-E\phi\big{(}u_N(t)\big{)}\Big{]}dt\right|
\le& \frac1T\int_0^TC(e^{-\alpha t}+\frac1{\alpha})N^{-1}dt
%\le& C\left(\frac1{\alpha T}+\frac1{\alpha}\right)N^{-1}
\le \frac C{\alpha}N^{-1}.
\end{align}

Using (\ref{ergodic}), (\ref{ergodic2}) and (\ref{ergodic3}), we have
\begin{align}
&\left|\int_{L^2}\phi(y)d\mu(y)-\int_{V_N}\phi(y)d\mu_N(y)\right|\nonumber\\
\le&\left|\frac{1}{T}\int_0^TE\phi\big{(}u(t)\big{)}dt-\int_{L^2}\phi(y)d\mu(y)\right|\nonumber\\
&+\left|\frac{1}{T}\int_0^TE\phi\big{(}u_N(t)\big{)}dt-\int_{V_N}\phi(y)d\mu_N(y)\right|\nonumber\\
&+\left|\frac{1}{T}\int_0^T\Big{[}E\phi\big{(}u(t)\big{)}-E\phi\big{(}u_N(t)\big{)}\Big{]}dt\right|\nonumber\\
\le&\frac C{\alpha}N^{-1} \le CN^{-1},
\end{align}
which implies that $\mu_N$ is a proper approximation of $\mu$. Thus, we give the following theorem.
\fi
\begin{tm}
Assume that $u_0\in\dot{H}^2$ and $\|Q^{\frac12}\|_{\mathcal{HS}(L^2,\dot{H}^3)}<\infty$. The error between invariant measures $\mu$ and $\mu_N$ is of order $2$, i.e.,
\begin{align*}
\left|\int_{L^2}\phi(y)d\mu(y)-\int_{V_N}\phi(y)d\mu_N(y)\right|<\frac C{\alpha}N^{-2}.
\end{align*}
\end{tm}
\begin{rk}
Although the time-independent weak error between $u$ and $u_N$ is obtained under the assumption $\|Q^{\frac12}\|_{\mathcal{HS}(L^2,\dot{H}^2)}<\infty$, it is necessary to assume in addition $\|Q^{\frac12}\|_{\mathcal{HS}(L^2,\dot{H}^3)}<\infty$ in order to get the unique ergodicity of $u$ (see \cite{debussche}).
\end{rk}

\vspace{5mm}
%%%%%%%%%%%%%%%%%%%%%%%%%%%%%%%%%%%%%%%%
\section{\textsc{\Large{F}ull discretization}}
In this section, 
 we discretize (\ref{spectral}) in temporal direction by a modification of the implicit Euler scheme to get a fully discrete scheme. We prove the ergodicity of the numerical solution $u_N^k$ of the fully discrete scheme, and get weak order $\frac{1}{2}$ of $u_N^k$ in temporal direction. 
 Thus, we achieve at least the same order as the weak error for the error of invariant measure, as a result of the time-independency of the weak error and the ergodicity of the solution.
%%%%%%%%%%%%%%%%%%%%%%%
\subsection{Fully discrete scheme}
We use a modified implicit Euler scheme to approximate (\ref{spectral}), and obtain the following scheme
\begin{equation}\label{BE}
\left\{
\begin{aligned}
&u_N^{k}-e^{-\alpha\tau}u_N^{k-1}=\left(\mathbf{i}\Delta u_N^{k}+\mathbf{i}\lambda\pi_N\left(\frac{|u_N^{k}|^2+|e^{-\alpha\tau}u_N^{k-1}|^2}{2}u_N^{k}\right)\right)\tau+\pi_NQ^{\frac{1}{2}}\delta W_{k}\\
&u_N^0=\pi_Nu_0(x),
\end{aligned}
\right.
\end{equation}
where $u_N^k$ is an approximation of $u_N(t_k)$, $\tau$ represents the uniform time step, $t_k=k\tau$, and $\delta W_{k}=W(t_{k})-W(t_{k-1})$.

The well-posedness of scheme (\ref{BE}), together with the uniform boundedness of the numerical solution, is stated in the following proposition. The time step $\tau$ is assumed to satisfy $\alpha\tau\in[0,1]$ in sequel.

\begin{prop}\label{ju}
Assume $u_0\in \dot{H}^0$. For sufficiently small $\tau$, there uniquely exists a family of $V_N$-valued and $\{\mathcal{F}_{t_k}\}_{k\in\mathbb{N}}$-adapted solutions $\{u_N^k\}_{k\in\mathbb{N}}$ of (\ref{BE}),  which satisfies that for any integer $p\geq 2$, there exists a constant $C=C(p,\alpha,u_N^0)>0$, such that
$$E\|u_N^k\|_0^p\leq C,\;\;\forall~k\in\mathbb{N}.$$
\end{prop}
\begin{proof}

\textbf{Step 1.} Existence and uniqueness of solution.

Similar to \cite{debouard}, we fix a family $\{g_k\}_{k\in\mathbb{N}}$ of deterministic functions in $V_N$. We also fix $\tilde{u}_N^{k-1}\in V_N$, the existence of solution $\tilde{u}_N^{k}\in V_N$ of
\begin{equation}\label{modify}
\tilde{u}_N^{k}-e^{-\alpha\tau}\tilde{u}_N^{k-1}=\mathbf{i}\tau\Delta \tilde{u}_N^{k}+\mathbf{i}\lambda\tau\pi_N\left(\frac{|\tilde{u}_N^{k}|^2+|e^{-\alpha\tau}\tilde{u}_N^{k-1}|^2}{2}\tilde{u}_N^{k}\right)+\sqrt{\tau}g_{k}
\end{equation}
can be proved by using Brouwer fixed point theorem. Indeed,
multiplying (\ref{modify}) by $\overline{\tilde{u}}_N^{k}$, integrating with respect to $x$ and taking the real part, we get
\begin{align*}
&\|\tilde{u}_N^{k}\|_0^2+\|\tilde{u}_N^{k}-e^{-\alpha\tau}\tilde{u}_N^{k-1}\|_0^2-e^{-2\alpha\tau}\|\tilde{u}_N^{k-1}\|_0^2\\
=&2\sqrt{\tau}Re\left[\int_0^1(\overline{\tilde{u}}_N^{k}-e^{-\alpha\tau}\overline{\tilde{u}}_N^{k-1})g_{k}dx+\int_0^1(e^{-\alpha\tau}\overline{\tilde{u}}_N^{k-1})g_{k}dx\right]\\
\leq&\|\tilde{u}_N^{k}-e^{-\alpha\tau}\tilde{u}_N^{k-1}\|_0^2+e^{-2\alpha\tau}\|\tilde{u}_N^{k-1}\|_0^2+2\tau\|g_{k}\|_0^2.
\end{align*}
Thus, 
\begin{align}\label{bound}
\|\tilde{u}_N^{k}\|_0^2\leq 2e^{-2\alpha\tau}\|\tilde{u}_N^{k-1}\|_0^2+2\tau\|g_k\|_0^2.
\end{align}
Define
\begin{align*}
\Lambda: V_N\times V_N&\rightarrow\mathcal{P}(L^2),\\
(\tilde{u}_N^{k-1},g_{k})&\mapsto \{\tilde{u}_N^{k}|\tilde{u}_N^{k}~\text{are~solutions~of}~(\ref{modify})\},
\end{align*}
where $\mathcal{P}(L^2)$ is the power set of $L^2$.
(\ref{bound}) implies that $\Lambda$ is continuous, and its graph is closed by the closed graph theorem.
When the spaces are endowed with their Borel $\sigma$-algebras, there is a measurable continuous function $\kappa: V_N\times V_N\rightarrow L^2$ such that
$$\kappa(u,g)\in\Lambda(u,g),~\forall~(u,g)\in V_N\times V_N.$$
Assume that $u_N^{k-1}\in V_N$ is a $\mathcal{F}_{t_{k-1}}$-measurable random variable, then $u_N^k=\kappa(u_N^{k-1},\frac{\pi_NQ^{\frac12}\delta W_{k}}{\sqrt{\tau}})$ is an $L^2$-valued solution of (\ref{BE}). Moreover, 
\begin{align*}
(1-\mathbf{i}\Delta\tau) u_N^{k}=e^{-\alpha\tau}u_N^{k-1}+\mathbf{i}\lambda\tau\pi_N\left(\frac{|u_N^{k}|^2+|e^{-\alpha\tau}u_N^{k-1}|^2}{2}u_N^{k}\right)+\pi_NQ^{\frac{1}{2}}\delta W_{k}\in V_N.
\end{align*}
Hence, $u_N^k$ is actually a $V_N$-valued solution of (\ref{BE}).

For any given $u_N^{k-1}$ and sufficiently small time step $\tau$, the solution $u_N^k$ is unique, which can be proved in a similar procedure as \cite{akrivis}. This fact will be used in proving the ergodicity of the numerical solution $\{u_N^k\}_{k\in\mathbb{N}}$, and it can be found in appendix \ref{5.2}.

\textbf{Step 2.} Boundedness of the $p$-moments.

The constants $C$ below may be different, but do not depend on time.

\romannumeral1) $p=2$. To show the boundedness, we multiply (\ref{BE}) by $\overline{u}_N^{k}$, integrate in [0,1] with respect to the space variable, take expectation and take the real part,

\begin{align}
&E\|u_N^{k}\|_0^2+E\|u_N^{k}-e^{-\alpha\tau}u_N^{k-1}\|_0^2-e^{-2\alpha\tau}E\|u_N^{k-1}\|_0^2\nonumber\\
=&2ReE\int_0^1\overline{u}_N^{k}\pi_NQ^{\frac12}\delta W_{k}dx\nonumber\\
=&2ReE\int_0^1\big{(}\overline{u}_N^{k}-e^{-\alpha\tau}\overline{u}_N^{k-1}\big{)}\pi_NQ^{\frac12}\delta W_{k}dx\nonumber\\
\leq&E\|u_N^{k}-e^{-\alpha\tau}u_N^{k-1}\|_0^2+E\|\pi_NQ^{\frac12}\delta W_{k}\|_0^2.\nonumber
\end{align}
It derives
\begin{align*}
E\|u_N^{k}\|_0^2\leq& e^{-2\alpha\tau}E\|u_N^{k-1}\|_0^2+C\tau\\
\leq&e^{-2\alpha\tau k}E\|u_N^0\|_0^2+C\tau(1+e^{-2\alpha\tau}+\cdots+e^{-2\alpha\tau(k-1)})\\
\leq&e^{-2\alpha t_k}E\|u_N^0\|_0^2+\frac{C\tau}{1-e^{-2\alpha\tau}}\\
\leq &E\|u_N^0\|_0^2+\frac{C}{e^{-1}2\alpha}
\end{align*}
for $\tau<\frac1{\alpha}$, where we have used $e^{-2\alpha\tau}<1-e^{-1}2\alpha\tau$ for $\tau<\frac1{\alpha}$.

\romannumeral2) $p=4$. In the case when p=2, without taking expectation, we have
$$\|u_N^{k}\|_0^2-e^{-2\alpha\tau}\|u_N^{k-1}\|_0^2+\|u_N^{k}-e^{-\alpha\tau}u_N^{k-1}\|_0^2=2Re\int_0^1\overline{u}_N^{k}\pi_NQ^{\frac12}\delta W_{k}dx.$$
Multiply both sides by $\|u_N^{k}\|_0^2$, take expectation and take the real part and we get
\begin{align*}
(LHS)=&E\|u_N^{k}\|_0^4-e^{-2\alpha\tau}E\|u_N^{k-1}\|_0^2\|u_N^{k}\|_0^2+E\Big{[}\|u_N^{k}-e^{-\alpha\tau}u_N^{k-1}\|_0^2\|u_N^{k}\|_0^2\Big{]}\nonumber\\
=&\frac12\Big{(}E\|u_N^{k}\|_0^4-e^{-4\alpha\tau}E\|u_N^{k-1}\|_0^4\Big{)}+\frac12E\Big{(}\|u_N^{k}\|_0^2-e^{-2\alpha\tau}\|u_N^{k-1}\|_0^2\Big{)}^2\nonumber\\
&+E\Big{[}\|u_N^{k}-e^{-2\alpha\tau}u_N^{k-1}\|_0^2\|u_N^{k}\|_0^2\Big{]}
\end{align*}
and
\begin{align*}
(RHS)=&2ReE\int_0^1\|u_N^{k}\|_0^2\overline{u}_N^{k}\pi_NQ^{\frac12}\delta W_{k}dx\nonumber\\
=&2Re E\int_0^1\left(\|u_N^{k}\|_0^2\big{(}\overline{u}_N^{k}-e^{-\alpha\tau}\overline{u}_N^{k-1}\big{)}\right)\pi_NQ^{\frac12}\delta W_{k}dx\nonumber\\
&+2ReE\int_0^1\Big{(}\big{(}\|u_N^{k}\|_0^2-e^{-2\alpha\tau}\|u_N^{k-1}\|_0^2\big{)}e^{-\alpha\tau}\overline{u}_N^{k-1}\Big{)}\pi_NQ^{\frac12}\delta W_{k}dx\nonumber\\
\leq&E\Big{[}\|u_N^{k}-e^{-\alpha\tau}u_N^{k-1}\|_0^2\|u_N^{k}\|_0^2\Big{]}+E\Big{(}\|u_N^{k}\|_0^2\|\pi_NQ^{\frac12}\delta W_{k}\|_0^2\Big{)}\nonumber\\
&+\frac14E\Big{(}\|u_N^{k}\|_0^2-e^{-2\alpha\tau}\|u_N^{k-1}\|_0^2\Big{)}^2+4e^{-2\alpha\tau}E\|\overline{u}_N^{k-1}\pi_NQ^{\frac12}\delta W_{k}\|_0^2\nonumber\\
\leq&E\Big{[}\|u_N^{k}-e^{-\alpha\tau}u_N^{k-1}\|_0^2\|u_N^{k}\|_0^2\Big{]}
+\frac12E\Big{(}\|u_N^{k}\|_0^2-e^{-2\alpha\tau}\|u_N^{k-1}\|_0^2\Big{)}^2+C\tau.\nonumber
\end{align*}
Compare (LHS) with (RHS), we obtain
$$E\|u_N^{k}\|_0^4\leq e^{-4\alpha\tau}E\|u_N^{k-1}\|_0^4+C\tau\leq C.$$

\romannumeral3) $p=3$. Using 1) and 2), it is easy to check that the following holds true
$$E\|u_N^k\|_0^3\leq E\frac{\|u_N^k\|_0^2+\|u_N^k\|_0^4}{2}\leq C.$$

\romannumeral4) $p>4$. By repeating above procedure, we complete the proof.
\end{proof}

Before showing the weak error between $u_N(t)$ and $u_N^k$, we need some a priori estimates on $\|u_N^k\|_1$ and $\|u_N^k\|_2$.
\begin{prop}\label{lambda-1}
Assume that $\lambda=0$ or $-1$, $u_0\in \dot{H}^1,~u_N^0=\pi_Nu_0$ and $\|Q^{\frac12}\|_{\mathcal{HS}(L^2,\dot{H}^1)}<\infty$.  
Then for any $p\geq 1$, there exists a constant $C=C(\alpha,u_0,p)$ independent of $N$ and $t_k$, such that
$$E\mathcal{H}^p_k\leq C,~\forall~k\in\mathbb{N},$$
where $\mathcal{H}_k:=\|\nabla u_N^k\|_0^2-\frac{\lambda}2\|u_N^k\|_{L^4}^4$.
\end{prop}

\begin{proof}
The proof for $\lambda=0$ is in the same procedure as that for $\lambda=-1$ and is much easier. Here we only give the proof for $\lambda=-1$ 
\begin{equation}\label{geshi-1}
u_N^k-e^{-\alpha\tau}u_N^{k-1}=\left(\mathbf{i}\Delta u_N^k-\mathbf{i}\pi_N\left(\frac{|u_N^k|^2+|e^{-\alpha\tau}u_N^{k-1}|^2}2u_N^k\right)\right)\tau+\pi_NQ^{\frac12}\delta W_k.
\end{equation}
\romannumeral1) $ p=1$.
Multiplying \eqref{geshi-1} by $\overline{u}_N^k-e^{-\alpha\tau}\overline{u}_N^{k-1}$, integrating with respect to $x$, taking the imaginary part and using the fact
$\big{(}(Id-\pi_N)v,v_N\big{)}=0,~\forall~v\in\dot{H}^0,~v_N\in V_N,$
 we have
\begin{align*}
&\|\nabla u_N^k\|_0^2+\|\nabla(u_N^k-e^{-\alpha\tau}u_N^{k-1})\|_0^2-e^{-2\alpha\tau}\|\nabla u_N^{k-1}\|_0^2\\
=&-Re\int_0^1\Big{(}|u_N^k|^2+|e^{-\alpha\tau}u_N^{k-1}|^2\Big{)}u_N^k(\overline{u}_N^k-e^{-\alpha\tau}\overline{u}_N^{k-1})dx\\
&+\frac2{\tau}Im\int_0^1\pi_NQ^{\frac12}\delta W_k(\overline{u}_N^k-e^{-\alpha\tau}\overline{u}_N^{k-1})dx\\
=:&A+B.
\end{align*}
Simple computations yield
\begin{align*}
A=&-Re\left[\int_0^1\Big{(}|u_N^k|^2+|e^{-\alpha\tau}u_N^{k-1}|^2\Big{)}\Big{(}\frac{u_N^k+e^{-\alpha\tau}u_N^{k-1}}2+\frac{u_N^k-e^{-\alpha\tau}u_N^{k-1}}2\Big{)}(\overline{u}_N^k-e^{-\alpha\tau}\overline{u}_N^{k-1})dx\right]\\
\leq&-\frac12\|u_N^k\|_{L^4}^4+\frac12e^{-4\alpha\tau}\|u_N^{k-1}\|_{L^4}^4\\
\leq&-\frac12\|u_N^k\|_{L^4}^4+\frac12e^{-2\alpha\tau}\|u_N^{k-1}\|_{L^4}^4
\end{align*}
and
\begin{align*}
B=&\frac2{\tau}Im\left[\int_0^1\pi_NQ^{\frac12}\delta W_k\bigg{[}-\mathbf{i}\tau\Delta\overline{u}_N^k+\mathbf{i}\tau\frac{|u_N^k|^2+|e^{-\alpha\tau}u_N^{k-1}|^2}2\overline{u}_N^k+\overline{\pi_NQ^{\frac12}\delta W_k}\bigg{]}dx\right]\\
=&2Re\left[\int_0^1\nabla(\pi_NQ^{\frac12}\delta W_k)\cdot\nabla\Big{(}\overline{u}_N^k-e^{-\alpha\tau}\overline{u}_N^{k-1}\Big{)}dx\right]+2Re\left[\int_0^1\nabla(\pi_NQ^{\frac12}\delta W_k)\cdot\nabla\Big{(}e^{-\alpha\tau}\overline{u}_N^{k-1}\Big{)}dx\right]\\
&+Re\left[\int_0^1\Big{(}|u_N^k|^2+|e^{-\alpha\tau}u_N^{k-1}|^2\Big{)}\overline{u}_N^k\cdot\pi_NQ^{\frac12}\delta W_kdx\right]\\
\leq&\frac14\|\nabla(u_N^k-e^{-\alpha\tau}u_N^{k-1})\|_0^2+C\|\nabla(\pi_NQ^{\frac12}\delta W_k)\|_0^2+2Re\left[\int_0^1\nabla(\pi_NQ^{\frac12}\delta W_k)\cdot\nabla\Big{(}e^{-\alpha\tau}\overline{u}_N^{k-1}\Big{)}dx\right]\\
&+Re\left[\int_0^1\Big{(}|u_N^k|^2+|e^{-\alpha\tau}u_N^{k-1}|^2\Big{)}\overline{u}_N^k\cdot\pi_NQ^{\frac12}\delta W_kdx\right].
\end{align*}
$ $\\
Denote $\mathcal{H}_k=\|\nabla u_N^k\|_0^2+\frac12\|u_N^k\|_{L^4}^4$, then
\begin{align}
&E\mathcal{H}_k+\frac34E\|\nabla(u_N^k-e^{-\alpha\tau}u_N^{k-1})\|_0^2\nonumber\\\label{Hditui}
\leq&e^{-2\alpha\tau}E\mathcal{H}_{k-1}+C\tau+ReE\left[\int_0^1\Big{(}|u_N^k|^2+|e^{-\alpha\tau}u_N^{k-1}|^2\Big{)}\overline{u}_N^k\cdot\pi_NQ^{\frac12}\delta W_kdx\right].
\end{align}
Based on the formula
$$(|a|^2+|b|^2)\overline{a}=\overline{a}|a-b|^2+b(\overline{a}-\overline{b})^2+3|b|^2(\overline{a}-\overline{b})+\overline{b}|a-b|^2+(\overline{b})^2(a-b)+2|b|^2\overline{b},$$
the last term on the right hand side can be rewritten as
\begin{align*}
&ReE\left[\int_0^1\Big{(}|u_N^k|^2+|e^{-\alpha\tau}u_N^{k-1}|^2\Big{)}\overline{u}_N^k\cdot\pi_NQ^{\frac12}\delta W_kdx\right]\\
=&ReE\int_0^1\overline{u}_N^k\Big{|}u_N^k-e^{-\alpha\tau}u_N^{k-1}\Big{|}^2\pi_NQ^{\frac12}\delta W_kdx+ReE\int_0^1e^{-\alpha\tau}u_N^{k-1}\big{(}\overline{u}_N^k-e^{-\alpha\tau}\overline{u}_N^{k-1}\big{)}^2\pi_NQ^{\frac12}\delta W_kdx\\
&+3ReE\int_0^1|e^{-\alpha\tau}u_N^{k-1}|^2\big{(}\overline{u}_N^k-e^{-\alpha\tau}\overline{u}_N^{k-1}\big{)}\pi_NQ^{\frac12}\delta W_kdx\\
&+ReE\int_0^1e^{-\alpha\tau}\overline{u}_N^{k-1}\Big{|}u_N^k-e^{-\alpha\tau}u_N^{k-1}\Big{|}^2\pi_NQ^{\frac12}\delta W_kdx\\
&+ReE\int_0^1(e^{-\alpha\tau}\overline{u}_N^{k-1})^2(u_N^k-e^{-\alpha\tau}u_N^{k-1})\pi_NQ^{\frac12}\delta W_kdx+2ReE\int_0^1|e^{-\alpha\tau}u_N^{k-1}|^2e^{-\alpha\tau}\overline{u}_N^{k-1}\pi_NQ^{\frac12}\delta W_kdx\\
=:&~a+b+c+d+e+f.
\end{align*}
Noting that $f=0$, it suffices to estimate the other five terms \begin{align*}
a+b+d\leq&E\Big{[}\|u_N^k\|_0\|u_N^k-e^{-\alpha\tau}u_N^{k-1}\|_{L^4}^2\|\pi_NQ^{\frac12}\delta W_k\|_{L^{\infty}}\\
&+2\|e^{-\alpha\tau}u_N^{k-1}\|_0\|u_N^k-e^{-\alpha\tau}u_N^{k-1}\|_{L^4}^2\|\pi_NQ^{\frac12}\delta W_k\|_{L^{\infty}}\Big{]}\\
\leq&E\left[\Big{(}\|u_N^k\|_0+2\|e^{-\alpha\tau}u_N^{k-1}\|_0\Big{)}\|\nabla(u_N^k-e^{-\alpha\tau}u_N^{k-1})\|_0^{\frac12}\|u_N^k-e^{-\alpha\tau}u_N^{k-1}\|_0^{\frac32}\|\pi_NQ^{\frac12}\delta W_k\|_{L^{\infty}}\right]\\
\leq&\frac14E\left[\|\nabla(u_N^k-e^{-\alpha\tau}u_N^{k-1})\|_0\|u_N^k-e^{-\alpha\tau}u_N^{k-1}\|_0\right]\\
&+CE\left[\Big{(}\|u_N^k\|_0^2+\|e^{-\alpha\tau}u_N^{k-1}\|_0^2\Big{)}\|u_N^k-e^{-\alpha\tau}u_N^{k-1}\|_0^2\|\pi_NQ^{\frac12}\delta W_k\|_{L^{\infty}}^2\right]\\
\leq&\frac14E\|\nabla(u_N^k-e^{-\alpha\tau}u_N^{k-1})\|_0^2+CE\left(\tau^{\frac12}\Big{(}\|u_N^k\|_0^2+\|e^{-\alpha\tau}u_N^{k-1}\|_0^2\Big{)}\|u_N^k-e^{-\alpha\tau}u_N^{k-1}\|_0^2\right)^2\\
&+CE\left(\tau^{-\frac12}\|\pi_NQ^{\frac12}\delta W_k\|^2_{L^{\infty}}\right)^2\\
\leq&\frac14E\|\nabla(u_N^k-e^{-\alpha\tau}u_N^{k-1})\|_0^2+C\tau,
\end{align*}
where in the last step we have used Proposition \ref{ju}, 
\begin{align*}
c+e\leq&4E\left[\|e^{-\alpha\tau}u_N^{k-1}\|_{L^4}^2\|u_N^k-e^{-\alpha\tau}u_N^{k-1}\|_0\|\pi_NQ^{\frac12}\delta W_k\|_{L^{\infty}}\right]\\
\leq&\frac12E\|u_N^k-e^{-\alpha\tau}u_N^{k-1}\|_0^2+8\eta\tau e^{-4\alpha\tau}E\|u_N^{k-1}\|_{L^4}^4\\
\leq&\frac12E\|u_N^k-e^{-\alpha\tau}u_N^{k-1}\|_0^2+2E\left[\left(\sqrt{\alpha}\tau^{\frac12}e^{-\alpha\tau}\|\nabla u_N^{k-1}\|_0\right)\left(\frac{C}{2\sqrt{\alpha}}8\eta\tau^{\frac12}e^{-3\alpha\tau}\|u_N^{k-1}\|_0^3\right)\right]\\
\leq&\frac12E\|u_N^k-e^{-\alpha\tau}u_N^{k-1}\|_0^2+\alpha\tau e^{-2\alpha\tau}E\|\nabla u_N^{k-1}\|_0^2+C\tau.
\end{align*}
Then (\ref{Hditui}) turns to be 
\begin{align*}
E\mathcal{H}_k\leq(1+\alpha\tau)e^{-2\alpha\tau}E\mathcal{H}_{k-1}+C\tau\leq e^{-\alpha\tau}E\mathcal{H}_{k-1}+C\tau.
\end{align*}
We finally obtain that
$$E\mathcal{H}_k\leq C.$$
\romannumeral2) $p=2$. From the case $p=1$, by $\|\cdot\|_{L^4}^4\leq\|\nabla\cdot\|_0\|\cdot\|_0^3$,  we get
\begin{align*}
\mathcal{H}_{k}-e^{-2\alpha\tau}\mathcal{H}_{k-1}
\leq&C\|\nabla(\pi_NQ^{\frac12}\delta W_k)\|_0^2+CRe\left[\int_0^1\nabla(\pi_NQ^{\frac12}\delta W_k)\cdot\nabla\Big{(}e^{-\alpha\tau}\overline{u}_N^{k-1}\Big{)}dx\right]\\
&+C\left(\tau^{\frac12}\Big{(}\|u_N^k\|_0^2+\|e^{-\alpha\tau}u_N^{k-1}\|_0^2\Big{)}\|u_N^k-e^{-\alpha\tau}u_N^{k-1}\|_0^2\right)^2\\&+C\left(\tau^{-\frac12}\|\pi_NQ^{\frac12}\delta W_k\|^2_{L^{\infty}}\right)^2
+\alpha\tau e^{-2\alpha\tau}\mathcal{H}_{k-1}+C\tau^{-1}\|u_N^{k-1}\|_0^6\|\pi_NQ^{\frac12}\delta W_k\|_{L^{\infty}}^4.
\end{align*} 
Multiplying above formula by $\mathcal{H}_{k}$, we have
\begin{align*}
&\mathcal{H}_{k}^2+(\mathcal{H}_{k}-e^{-2\alpha\tau}\mathcal{H}_{k-1})^2-e^{-4\alpha\tau}\mathcal{H}_{k-1}^2\\
\leq&C\mathcal{H}_{k}\|\nabla(\pi_NQ^{\frac12}\delta W_k)\|_0^2+C\mathcal{H}_{k}Re\left[\int_0^1\nabla(\pi_NQ^{\frac12}\delta W_k)\cdot\nabla\Big{(}e^{-\alpha\tau}\overline{u}_N^{k-1}\Big{)}dx\right]\\
&+C\tau\mathcal{H}_{k}\Big{(}\|u_N^k\|_0^2+\|e^{-\alpha\tau}u_N^{k-1}\|_0^2\Big{)}^2\|u_N^k-e^{-\alpha\tau}u_N^{k-1}\|_0^4\\
&+C\mathcal{H}_{k}\left(\tau^{-\frac12}\|\pi_NQ^{\frac12}\delta W_k\|^2_{L^{\infty}}\right)^2
+\alpha\tau e^{-2\alpha\tau}\mathcal{H}_{k}\mathcal{H}_{k-1}+C\tau^{-1}\mathcal{H}_{k}\|u_N^{k-1}\|_0^6\|\pi_NQ^{\frac12}\delta W_k\|_{L^{\infty}}^4\\
=:&a'+b'+c'+d'+e'+f',
\end{align*}
where
\begin{align*}
E[a'+b'+c'+d']\leq&\frac14E(\mathcal{H}_{k}-e^{-2\alpha\tau}\mathcal{H}_{k-1})^2+C\tau\\
&+C\tau e^{-2\alpha\tau}E\left[\mathcal{H}_{k-1}\Big{(}\|u_N^k\|_0^2+\|e^{-\alpha\tau}u_N^{k-1}\|_0^2\Big{)}^2\|u_N^k-e^{-\alpha\tau}u_N^{k-1}\|_0^4\right]\\
\leq&\frac14E(\mathcal{H}_{k}-e^{-2\alpha\tau}\mathcal{H}_{k-1})^2+\frac12\tau e^{-4\alpha\tau}E\mathcal{H}_{k-1}^2+C\tau,
\end{align*}
\begin{align*}
E[e']\leq&\frac12E\left(\mathcal{H}_k-e^{-2\alpha\tau}\mathcal{H}_{k-1}\right)^2+(\frac12\alpha^2\tau^2+\alpha\tau)e^{-4\alpha\tau}E\mathcal{H}_{k-1}^2\\
\leq&\frac12E\left(\mathcal{H}_k-e^{-2\alpha\tau}\mathcal{H}_{k-1}\right)^2+\frac32\alpha\tau e^{-4\alpha\tau}E\mathcal{H}_{k-1}^2
\end{align*}
and
\begin{align*}
E[f']\leq&\frac14E\left(\mathcal{H}_k-e^{-2\alpha\tau}\mathcal{H}_{k-1}\right)^2+C\tau^{-2}E\left[\|u_N^{k-1}\|_0^{12}\|\pi_NQ^{\frac12}\delta W_k\|_{L^{\infty}}^8\right]\\
&+\alpha\tau e^{-4\alpha\tau}E\mathcal{H}_{k-1}^2+C\tau^{-3}E\left[\|u_N^{k-1}\|_0^{12}\|\pi_NQ^{\frac12}\delta W_k\|_{L^{\infty}}^8\right]\\
\leq&\frac14E\left(\mathcal{H}_k-e^{-2\alpha\tau}\mathcal{H}_{k-1}\right)^2+\alpha\tau e^{-4\alpha\tau}E\mathcal{H}_{k-1}^2+C\tau.
\end{align*}
Then we conclude
\begin{align*}
E\mathcal{H}_{k}^2\leq (1+3\alpha\tau)e^{-4\alpha\tau}E\mathcal{H}_{k-1}^2+C\tau
\leq e^{-\alpha\tau}E\mathcal{H}_{k-1}^2+C\tau\leq C,
\end{align*}
where we have used $(1+3\alpha\tau)e^{-3\alpha\tau}\leq1$ for $\alpha\tau<1$.\\
\romannumeral3) For $p=2^l,\, l\in\mathbb{N}$, the result can be proved by above procedure. So it also holds for any $p\in\mathbb{N}$.
\end{proof}

\begin{cor}
Under the assumptions in Proposition \ref{lambda-1}, we have
$$E\|u_N^k-e^{-\alpha\tau}u_N^{k-1}\|_0^{2p}\leq C\tau^p,$$
where constant $C$ is independent of $N$ and $t_k$. 
\end{cor}
\begin{proof}
It is easy to check this by multiplying $\overline{u}_N^k-e^{-\alpha\tau}\overline{u}_N^{k-1}$ to both sides of (\ref{geshi-1}), integrating with respect to $x$ and taking expectation,
\begin{align*}
E\|u_N^k-e^{-\alpha\tau}u_N^{k-1}\|_0^{2p}=&E\Bigg{[}\tau Im\int_0^1\nabla u_N^k\nabla(\overline{u}_N^k-e^{-\alpha\tau}\overline{u}_N^{k-1})dx\\
&+\frac{\tau}4Im\int_0^1\left(|u_N^k|^2+|e^{-\alpha\tau}u_N^{k-1}|^2\right)\left(u_N^k+e^{-\alpha\tau}u_N^{k-1}\right)\left(\overline{u}_N^k-e^{-\alpha\tau}\overline{u}_N^{k-1}\right)dx\\
&+Re\int_0^1\pi_NQ^{\frac12}\delta W_k\left(\overline{u}_N^k-e^{-\alpha\tau}\overline{u}_N^{k-1}\right)dx\Bigg{]}^p\\
\leq&CE\Bigg{[}\tau^p\|\nabla u_N^k\|_0^p\|\nabla\left(u_N^k-e^{-\alpha\tau}u_N^{k-1}\right)\|_0^p\\
&+\tau^p\left(\|u_N^k\|_1^{2p}+\|u_N^{k-1}\|_1^{2p}\right)\left(\|u_N^k\|_0^{2p}+\|u_N^{k-1}\|_0^{2p}\right)\Bigg{]}\\
&+CE\|\pi_NQ^{\frac12}\delta W_k\|_0^{2p}+\frac12E\|u_N^k-e^{-\alpha\tau}u_N^{k-1}\|_0^{2p}\\
\leq&\frac12E\|u_N^k-e^{-\alpha\tau}u_N^{k-1}\|_0^{2p}+C\tau^p.
\end{align*}
Then we complete the proof by Proposition \ref{lambda-1}.
\end{proof}

\begin{prop}
Under the assumptions $\lambda=0$ or $-1$, $u_0\in\dot{H}^2$ and $\|Q^{\frac12}\|_{\mathcal{HS}(L^2,\dot{H}^2)}<\infty$, we also have the uniform boundedness of 2-norm as follows
$$E\|u_N^k\|_2^2\leq C,~\forall~k\in\mathbb{N},$$
where $C$ is also independent of $N$ and $t_k$.
\end{prop}
\begin{proof}
We also give the proof for $\lambda=-1$ only. 
Multiply (\ref{geshi-1}) by $\Delta(\overline{u}_N^k-e^{-\alpha\tau}\overline{u}_N^{k-1})$, integrating with respect to $x$, and then taking the imaginary part, we obtain
\begin{align*}
&\|\Delta u_N^k\|_0^2+\|\Delta(u_N^k-e^{-\alpha\tau}u_N^{k-1})\|_0^2-e^{-2\alpha\tau}\|\Delta u_N^{k-1}\|_0^2\\
=&Re\int_0^1\left(|u_N^k|^2+|e^{-\alpha\tau}u_N^{k-1}|^2\right)u_N^k\Delta(\overline{u}_N^k-e^{-\alpha\tau}\overline{u}_N^{k-1})dx\\
&-\frac2{\tau}Im\int_0^1\pi_NQ^{\frac12}\delta W_k\Delta(\overline{u}_N^k-e^{-\alpha\tau}\overline{u}_N^{k-1})dx\\
=:&A'+B'.
\end{align*}
According to the uniform boundedness of any order of 0-norm and 1-norm, we have the following estimations.
\begin{align*}
E[A']=&ReE\int_0^1|u_N^k|^2u_N^k\Delta(\overline{u}_N^k-e^{-\alpha\tau}\overline{u}_N^{k-1})dx+e^{-3\alpha\tau}ReE\int_0^1|u_N^{k-1}|^2u_N^{k-1}\Delta(\overline{u}_N^k-e^{-\alpha\tau}\overline{u}_N^{k-1})dx\\
&+e^{-2\alpha\tau}ReE\int_0^1|u_N^{k-1}|^2(u_N^k-e^{-\alpha\tau}u_N^{k-1})\Delta(\overline{u}_N^k-e^{-\alpha\tau}\overline{u}_N^{k-1})dx\\
=&ReE\int_0^1|u_N^k|^2u_N^k\Delta\overline{u}_N^kdx-e^{-4\alpha\tau}ReE\int_0^1|u_N^{k-1}|^2u_N^{k-1}\Delta\overline{u}_N^{k-1}dx\\
&+e^{-2\alpha\tau}ReE\int_0^1|u_N^{k-1}|^2(u_N^k-e^{-\alpha\tau}u_N^{k-1})\Delta(\overline{u}_N^k-e^{-\alpha\tau}\overline{u}_N^{k-1})dx\\
&+ReE\int_0^1u_N^k\Delta\overline{u}_N^k|u_N^k-e^{-\alpha\tau}u_N^{k-1}|^2dx\\
&+2ReE\int_0^1\overline{u}_N^k(\nabla u_N^k)^2(\overline{u}_N^k-e^{-\alpha\tau}\overline{u}_N^{k-1})dx\\
&+4ReE\int_0^1u_N^k|\nabla u_N^k|^2(\overline{u}_N^k-e^{-\alpha\tau}\overline{u}_N^{k-1})dx\\
&+ReE\int_0^1(u_N^k-e^{-\alpha\tau}u_N^{k-1})\Delta\overline{u}_N^k\left(|u_N^k|^2-|e^{-\alpha\tau}u_N^{k-1}|^2\right)dx\\
=:&A_a^k-e^{-4\alpha\tau}A_a^{k-1}+A_b+A_c+A_d+A_e+A_f.
\end{align*}
We estimate above terms repectively and obtain 
\begin{align*}
-e^{-4\alpha\tau}A_a^{k-1}=&-e^{-2\alpha\tau}A_a^{k-1}+e^{-2\alpha\tau}(1-e^{-2\alpha\tau})A_a^{k-1}\\
\leq&-e^{-2\alpha\tau}A_a^{k-1}+C\tau E\|u_N^{k-1}\|_1^4
\leq-e^{-2\alpha\tau}A_a^{k-1}+C\tau,
\end{align*}
\begin{align*}
A_b\leq&e^{-2\alpha\tau}E\left[\|u_N^{k-1}\|_{L^{\infty}}^2\|u_N^k-e^{-\alpha\tau}u_N^{k-1}\|_0\|\Delta(u_N^k-e^{-\alpha\tau}u_N^{k-1})\|_0\right]\\
\leq&\frac16E\|\Delta(u_N^k-e^{-\alpha\tau}u_N^{k-1})\|_0^2+C\tau E\|u_N^{k-1}\|_1^8+C\tau^{-1}E\|u_N^k-e^{-\alpha\tau}u_N^{k-1}\|_0^4\\
\leq&\frac16E\|\Delta(u_N^k-e^{-\alpha\tau}u_N^{k-1})\|_0^2+C\tau,
\end{align*}
\begin{align*}
A_c\leq&E\left[\|u_N^k-e^{-\alpha\tau}u_N^{k-1}\|_{L^4}^2\|u_N^k\|_{L^{\infty}}\|\Delta u_N^k\|_0\right]\\
\leq&C\tau^{-1}E\left[\|\nabla(u_N^k-e^{-\alpha\tau}u_N^{k-1})\|_0\|u_N^k-e^{-\alpha\tau}u_N^{k-1}\|_0^3\|u_N^k\|_1^2\right]+\frac18\alpha\tau E\|\Delta u_N^k\|_0^2\\
\leq&\frac16E\|\Delta(u_N^k-e^{-\alpha\tau}u_N^{k-1})\|_0^2+C\tau^{-5}E\|u_N^k-e^{-\alpha\tau}u_N^{k-1}\|_0^{12}+C\tau E\|u_N^k\|_1^8+\frac18\alpha\tau E\|\Delta u_N^k\|_0^2\\
\leq&\frac16E\|\Delta(u_N^k-e^{-\alpha\tau}u_N^{k-1})\|_0^2+\frac18\alpha\tau E\|\Delta u_N^k\|_0^2+C\tau,
\end{align*}
\begin{align*}
A_d=&2ReE\int_0^1\overline{u}_N^k(\nabla u_N^k)^2\bigg{[}-\mathbf{i}\tau\Delta\overline{u}_N^k+\mathbf{i}\tau\pi_N\bigg{(}\frac{|u_N^k|^2+|e^{-\alpha\tau}u_N^{k-1}|^2}2\overline{u}_N^k\bigg{)}+\overline{\pi_NQ^{\frac12}\delta W_k}\bigg{]}dx\\
\leq&\frac1{16}\alpha\tau E\|\Delta u_N^k\|_0^2+C\tau+2ReE\int_0^1\overline{u}_N^k(\nabla u_N^k)^2\overline{\pi_NQ^{\frac12}\delta W_k}dx\\
\leq&\frac1{16}\alpha\tau E\|\Delta u_N^k\|_0^2+C\tau+2ReE\int_0^1(\overline{u}_N^k-e^{-\alpha\tau}\overline{u}_N^{k-1})(\nabla u_N^k)^2\overline{\pi_NQ^{\frac12}\delta W_k}dx\\
&+2ReE\int_0^1e^{-\alpha\tau}\overline{u}_N^{k-1}\left((\nabla u_N^k)^2-(e^{-\alpha\tau}\nabla u_N^{k-1})^2\right)\overline{\pi_NQ^{\frac12}\delta W_k}dx\\
\leq&\frac1{16}\alpha\tau E\|\Delta u_N^k\|_0^2+C\tau+CE\left[\|u_N^k-e^{-\alpha\tau}u_N^{k-1}\|_0\|\nabla u_N^k\|_{L^4}^2\|\pi_NQ^{\frac12}\delta W_k\|_{L^{\infty}}\right]\\
&+CE\left[\|\nabla(u_N^k-e^{-\alpha\tau}u_N^{k-1})\|_0\left(\|u_N^{k-1}\|_1\|u_N^{k}\|_1+\|u_N^{k-1}\|_1^2\right)\|\pi_NQ^{\frac12}\delta W_k\|_{L^{\infty}}\right]\\
\leq&\frac16E\|\Delta(u_N^k-e^{-\alpha\tau}u_N^{k-1})\|_0^2+\frac18\alpha\tau E\|\Delta u_N^k\|_0^2+C\tau,
\end{align*}
and
\begin{align*}
A_f=&ReE\int_0^1(u_N^k-e^{-\alpha\tau}u_N^{k-1})\Delta\overline{u}_N^kRe\big{[}\left(u_N^k-e^{-\alpha\tau}u_N^{k-1}\right)(\overline{u}_N^k+e^{-\alpha\tau}\overline{u}_N^{k-1})\big{]}dx\\
\leq&E\left[\|u_N^k-e^{-\alpha\tau}u_N^{k-1}\|_{L^4}^2(\|u_N^k\|_{L^{\infty}}+\|u_N^{k-1}\|_{L^{\infty}})\|\Delta u_N^k\|_0\right]\\
\leq&\frac16E\|\Delta(u_N^k-e^{-\alpha\tau}u_N^{k-1})\|_0^2+\frac18\alpha\tau E\|\Delta u_N^k\|_0^2+C\tau,
\end{align*}
where $A_e$ has an same estimation as $A_d$ and we have used that $\|\nabla\cdot\|_0\cong\|\cdot\|_1\leq\|\cdot\|_2\cong\|\Delta\cdot\|_0$. 
So we obtain
$$E[A']\leq\frac56E\|\Delta(u_N^k-e^{-\alpha\tau}u_N^{k-1})\|_0^2+\frac12\alpha\tau E\|\Delta u_N^k\|_0^2+C\tau.$$
For term $B'$, we have
\begin{align*}
E[B']=&-\frac2{\tau}ImE\int_0^1\Delta\left(\pi_NQ^{\frac12}\delta W_k\right)\left(-\mathbf{i}\tau\Delta\overline{u}_N^k+\mathbf{i}\pi_N\left(\frac{|u_N^k|^2+|e^{-\alpha\tau}u_N^{k-1}|^2}2\overline{u}_N^k\right)\tau+\overline{\pi_NQ^{\frac12}\delta W_k}\right)dx\\
=&2ReE\int_0^1\Delta\left(\pi_NQ^{\frac12}\delta W_k\right)\Delta(\overline{u}_N^k-e^{-\alpha\tau}\overline{u}_N^{k-1})dx\\
&-ReE\int_0^1\Delta\left(\pi_NQ^{\frac12}\delta W_k\right)\left(|u_N^k|^2\overline{u}_N^k-|e^{-\alpha\tau}u_N^{k-1}|^2e^{-\alpha\tau}u_N^{k-1}\right)dx\\
&-ReE\int_0^1\Delta\left(\pi_NQ^{\frac12}\delta W_k\right)|e^{-\alpha\tau}u_N^{k-1}|^2(\overline{u}_N^k-e^{-\alpha\tau}\overline{u}_N^{k-1})dx\\
\leq&\frac16E\|\Delta(u_N^k-e^{-\alpha\tau}u_N^{k-1})\|_0^2+C\tau.
\end{align*}
Denoting $\mathcal{K}_k:=\|\Delta u_N^k\|_0^2-Re\int_0^1|u_N^k|^2u_N^k\Delta\overline{u}_N^kdx$, then
$
E\|\Delta u_N^k\|_0^2\leq E\mathcal{K}_k+C
$
and
\begin{align*}
E\mathcal{K}_k-e^{-2\alpha\tau}E\mathcal{K}_{k-1}\leq \frac12\alpha\tau E\|\Delta u_N^k\|_0^2+C\tau
\leq \frac12\alpha\tau E\mathcal{K}_k+C\tau.
\end{align*}
Finally, $$E\mathcal{K}_k\leq(1-\frac12\alpha\tau)^{-1}e^{-2\alpha\tau}E\mathcal{K}_{k-1}+C\tau\leq C,$$
where we have used $(1-\frac12\alpha\tau)^{-1}e^{-2\alpha\tau}\leq e^{-\alpha\tau}$ for $\alpha\tau<1$.
\end{proof}

%%%%%%%%%%%%%%%%%%%%%%%3.2
\subsection{Ergodicity of the fully discrete scheme}
To prove the ergodicity of the scheme (\ref{BE}), 
 we will use the discrete form of theorem \ref{ergodicity}. We give some existing results before our theorem.
\begin{ap}[Minorization condition in \cite{mattingly}]\label{ap1}
The Markov chain $(x_n)_{n\in\mathbb{N}}$ with transition kernel $P_n(x,G)=P(x_n\in G|x_0=x)$ satisfies, for some fixed compact set $\mathcal{C}\in\mathcal{B}(\mathbb{R}^d)$, the following:\\
\romannumeral1) for some $y^*\in int(\mathcal{C})$ there is, for any $\delta>0$, a $t_1=t_1(\delta)\in\mathbb{N}$ such that$$P_{t_1}(x,B_{\delta}(y^*))>0\quad\forall x\in\mathcal{C};$$
\romannumeral2) the transition kernel possesses a density $p_n(x,y)$, more precisely
$$P_n(x,G)=\int_Gp_n(x,y)dy\quad\forall x\in\mathcal{C},\,G\in\mathcal{B}(\mathbb{R}^d)\cap\mathcal{B}(\mathcal{C})$$
and $p_n(x,y)$ is jointly continuous in $(x,y)\in \mathcal{C}\times\mathcal{C}$.
\end{ap}

\begin{ap}[Lyapunov condition in \cite{mattingly}]\label{ap2}
There is a function $F : \mathbb{R}^d\to[1,\infty)$, with $\lim_{|x|\to\infty}F(x)=\infty$, real numbers $\theta\in(0,1)$, and $\gamma\in[0,\infty)$ such that
$$E[F(x_{n+1})|\mathcal{F}_n]\leq\theta F(x_n)+\gamma.$$
\end{ap}

\begin{df}
We say that function $F$ is essentially quadratic if there exist constants $C_i>0,\;i=1,2,3$, such that
$$C_1(1+\|x\|^2)\le F(x)\le C_2(1+\|x\|^2),\;\;|\nabla F(x)|\le C_3(1+\|x\|).$$ 
\end{df}

\begin{tm}[\cite{mattingly}]\label{mat}
Assume that a Markov chain $(x_n)_{n\in\N}$  satisfies Assumptions \ref{ap1} and \ref{ap2}  with an essentially quadratic $F$, then the chain possesses a unique invariant measure.
\end{tm}

Based on the preliminaries above and the theory of Markov chains, we prove the following theorem.

\begin{tm}
For all $\tau$ sufficiently small, the solution  $(u_N^k)_{k\in\mathbb{N}}$ of scheme (\ref{BE}) has a unique invariant measure $\mu_N^{\tau}$. Thus, it is ergodic.
\end{tm}
\begin{proof}

\romannumeral1) Lyapunov condition.
Based on Proposition \ref{ju}, we can take essentially quadratic function $F(\cdot)=1+\|\cdot\|_0^2$ as the Lyapunov function, and the Lyapunov condition holds.\\
\romannumeral2) Minorization condition.
In scheme (\ref{BE}), it gives
\begin{align}\label{scheme1}
P_N^{k}=&e^{-\alpha\tau}P_N^{k-1}-\tau\left(\Delta Q_N^{k}+\frac{\lambda}2\pi_N\Big{(}\left(|P_N^{k}|^2+|Q_N^{k}|^2+|e^{-\alpha\tau}P_N^{k-1}|^2+|e^{-\alpha\tau}Q_N^{k-1}|^2\right)Q_N^{k}\Big{)}\right)\nonumber\\
&+\sum_{m=1}^N\sqrt{\eta_m}e_m\delta_{k}\beta_{m}^1,\\\label{scheme2}
Q_N^{k}=&e^{-\alpha\tau}Q_N^{k-1}+\tau\left(\Delta P_N^{k}+\frac{\lambda}2\pi_N\Big{(}\left(|P_N^{k}|^2+|Q_N^{k}|^2+|e^{-\alpha\tau}P_N^{k-1}|^2+|e^{-\alpha\tau}Q_N^{k-1}|^2\right)P_N^{k}\Big{)}\right)\nonumber\\
&+\sum_{m=1}^N\sqrt{\eta_m}e_m\delta_{k}\beta_{m}^2,
\end{align}
where $P_N^k$ and $Q_N^k$ denote the real and imaginary part of $u_N^k$ respectively, that is $u_N^k=P_N^k+\mathbf{i}Q_N^k$. Also, $\pi_NQ^{\frac12}\delta W_{k}=\sum_{m=1}^N\sqrt{\eta_m}e_m\left(\delta_k\beta_{m}^1+\mathbf{i}\delta_k\beta_{m}^2\right)$, where $\delta_k\beta_{m}^1$ and $\delta_k\beta_{m}^2$ are the real and imaginary part of $\delta W_k$ respectively.

For any $y_1=a_1+\mathbf{i}b_1,\,y_2=a_2+\mathbf{i}b_2\in V_N$ with $a_i$ and $b_i$ denoting the real and imaginary part of $y_i$ ($i=1,2$) respectively, as $\{e_m\}_{m=1}^N$ is a basis of $V_N$,          $\{\delta_{k}\beta_{m}^1,~\delta_{k}\beta_{m}^2\}_{m=1}^N$ can be uniquely determined to ensure that $(P_N^{k-1},Q_N^{k-1})=(a_1,b_1)$ and $(P_N^{k},Q_N^{k})=(a_2,b_2)$, which implies the irreducibility of $u_N^k$.

As stated in Proposition \ref{ju}, the $\mathcal{F}_{t_k}$-measurable solution $\{u_N^k\}_{k\in\mathbb{N}}$ is defined through a unique continuous function: $u_N^k=\kappa(u_N^{k-1},\frac{\pi_NQ^{\frac12}\delta W_k}{\sqrt{\tau}})$, where $\delta W_k$ has a $C^{\infty}$ density. Thus, the transition kernel $P_1(x,G),~G\in\mathcal{B}(V_N)$ possesses a jointly continuous density $p_1(x,y)$.
Furthermore, densities $p_k(x,y)$ are achieved by the time-homogeneous property of Markov chain $\{u_N^k\}_{k\in\mathbb{N}}$.

With above conditions, based on Theorem \ref{mat}, we prove that $u_N^k$ possesses a unique invariant measure.
\end{proof}
%%%%%%%%%%%%%%%%%%%%%%%   
\subsection{Weak error between solutions $u_N$ and $u_N^k$}
We still use modified processes to calculate the weak error of the fully discrete scheme in temporal direction. 
Denote $S_{\tau}=(Id-\mathbf{i}\tau\Delta)^{-1}e^{-\alpha\tau}$, then scheme (\ref{BE}) is rewritten as
\begin{align}
u_N^k=&S_{\tau}u_N^{k-1}+\mathbf{i}\lambda\tau e^{\alpha\tau}S_{\tau}\pi_N\left(\frac{|u_N^k|^2+|e^{-\alpha\tau}u_N^{k-1}|^2}{2}u_N^k\right)+e^{\alpha\tau}S_{\tau}\pi_NQ^{\frac12}\delta W_k\nonumber\\
=&S_{\tau}^ku_N^0+\mathbf{i}\lambda\tau e^{\alpha\tau}\sum_{l=1}^kS_{\tau}^{k+1-l}\pi_N\left( \frac{|u_N^l|^2+|e^{-\alpha\tau}u_N^{l-1}|^2}{2}u_N^l\right)+e^{\alpha\tau}\sum_{l=1}^kS_{\tau}^{k+1-l}\pi_NQ^{\frac12}\delta W_l
\end{align}

\begin{lm}\label{operator2}
For any $k\in \mathbb{N}$ and sufficiently small $\tau$, we have the following estimates,
\begin{align*}
&\romannumeral1)~\|S_{\tau}^k-S(t)\|_{\mathcal{L}(\dot{H}^2,L^2)}\leq C(t+\tau)^{\frac12}e^{-\alpha t}\tau^{\frac12},\;\;\;t\in[t_{k-1},t_{k+1}],\\
&\romannumeral2)~\|S_{\tau}^k-S(t)\|_{\mathcal{L}(\dot{H}^1,\dot{H}^1)}\leq Ce^{-\alpha t},\;\;\;t\in[t_{k-1},t_{k+1}],
\end{align*}
where the constant $C=C(\alpha)$ is independent of $k$ and $\tau$.
\end{lm}
\begin{proof} 
\textbf{Step 1.} If $t=t_k$. As $S(t)$ is the operator semigroup of equation $du(t)=(\mathbf{i}\Delta-\alpha)u(t)dt,\,u(0)=u^0\in \dot{H}^2$, and $S_{\tau}$ is the corresponding discrete operator semigroup, we have
\begin{align}\label{use}
&S_{\tau}^ku(0)=u^k=e^{-\alpha\tau}u^{k-1}+\mathbf{i}\tau\Delta u^k,\\
&S(t_k)u(0)=u(t_k)=e^{-\alpha\tau}u(t_{k-1})+\int_{t_{k-1}}^{t_k}\mathbf{i}e^{-\alpha(t_k-s)}\Delta u(s)ds.
\end{align}
Denote $e_k=u^k-u(t_k)=\left(S_{\tau}^k-S(t_k)\right)u(0)$ with $e_0=0$, then
$$e_k=e^{-\alpha\tau}e_{k-1}+\mathbf{i}\tau\Delta e_k+\mathbf{i}\int_{t_{k-1}}^{t_k}\big{[}\Delta u(t_k)-e^{-\alpha(t_k-s)}\Delta u(s)\big{]}ds.$$\\
Multiply $\overline{e}_k$ to above formula, integrate with respect to $x$, take the real part, and we get
\begin{align*}
&\frac12\left[\|e_k\|_0^2+\|e_k-e^{-\alpha\tau}e_{k-1}\|_0^2-e^{-2\alpha\tau}\|e_{k-1}\|_0^2\right]\\
=&Re\left[\mathbf{i}\int_0^1\int_{t_{k-1}}^{t_k}\Delta\overline{e}_k\int^{t_k}_s\mathbf{i}e^{-\alpha(t_k-r)}\Delta u(r)drdsdx\right]\\
\leq&C\int_{t_{k-1}}^{t_k}\int_s^{t_k}\|\Delta u^k-\Delta u(t_{k})\|_0\|\Delta u(r)\|_0drds\\
\leq&Ce^{-2\alpha t_{k}}\|\Delta u(0)\|_0^2\tau^2,
\end{align*}
where we have used the fact that
$\|\Delta u^k\|^2_0\leq e^{-2\alpha t_k}\|\Delta u^0\|^2_0$ and 
$\|\Delta u(t)\|_0\leq Ce^{-\alpha t}\|\Delta u(0)\|_0.$
 In fact, multiplying $\Delta\overline{u}^k-e^{-\alpha\tau}\Delta \overline{u}^{k-1}$ to (\ref{use}), integrating in space and taking the imaginary part, we obtain
\begin{align*}
\|\Delta u^k\|^2_0\leq e^{-2\alpha\tau}\|\Delta u^{k-1}\|^2_0
\leq e^{-2\alpha t_k}\|\Delta u^0\|^2_0.
\end{align*}
Then it's easy to check that
$$\|e_k\|_0^2\leq e^{-2\alpha\tau}\|e_{k-1}\|_0^2+Ce^{-2\alpha t_{k}}\|\Delta u(0)\|_0^2\tau^2$$
 leads to
\begin{equation}\label{jieguo}
\|e_k\|_0^2\leq Ct_ke^{-2\alpha t_k}\|\Delta u(0)\|_0^2\tau,
\end{equation}
which finally yields
$\|S_{\tau}^k-S(t_k)\|_{\mathcal{L}(\dot{H}^2,L^2)}\leq Ct_k^{\frac12}e^{-\alpha t_k}\tau^{\frac12}$ in $\romannumeral1)$. 

For $\romannumeral2)$, we have 
\begin{align*}
\|\left(S_{\tau}^k-S(t_k)\right)u(0)\|_1^2
=&\sum_{n=1}^{\infty}\left|e^{-\alpha t_k}\left((1+n^2\pi^2)^{-k}-e^{-n^2\pi^2t_k}\right)(u(0),e_n)\right|^2|\lambda_n|\\
\le&4e^{-2\alpha t_k}\sum_{n=1}^{\infty}\left|(u(0),e_n)\right|^2|\lambda_n|=4e^{-2\alpha t_k}\|u(0)\|_1^2.
\end{align*}

In the following two steps, we only give the proof of $\romannumeral1)$, and $\romannumeral2)$ can be proved in a same procedure. We use the notation $\|\cdot\|=\|\cdot\|_{\mathcal{L}(\dot{H}^2,L^2)}$, which is an operator norm defined at the beginning of this paper.

\textbf{Step 2.} If $t\in[t_{k-1},t_k]$,
\begin{align}
\|S_{\tau}^k-S(t)\|\leq&\|S_{\tau}^k-S(t_k)\|+\|S(t_k)-S(t)\|\nonumber\\
\leq&Ct_k^{\frac12}e^{-\alpha t_k}\tau^{\frac12}+e^{-\alpha t}|e^{-\alpha(t_k-t)}-1|\nonumber\\
\leq&Ct_k^{\frac12}e^{-\alpha t_k}\tau^{\frac12}+e^{-\alpha t}\sum_{n=1}^{\infty}\frac{1}{n!}(\alpha\tau)^n\nonumber\\
\leq&Ct_k^{\frac12}e^{-\alpha t_k}\tau^{\frac12}+e^{-\alpha t}\alpha\tau\frac{e^{\alpha\tau}-1}{\alpha\tau}\nonumber\\
\leq&C(t+\tau)^{\frac12}e^{-\alpha t}\tau^{\frac12}.\nonumber
\end{align}
We have used the fact that $\frac{e^{\alpha\tau}-1}{\alpha\tau}$ is uniformly bounded for $\alpha\tau\in[0,1]$.

\textbf{Step 3.} If $t\in[t_k,t_{k+1}]$,
\begin{align}
\|S_{\tau}^k-S(t)\|\leq&\|S_{\tau}^k-S(t_k)\|+\|S(t_k)-S(t)\|\nonumber\\
\leq&Ct_k^{\frac12}e^{-\alpha t_k}\tau^{\frac12}+e^{-\alpha t}|e^{-\alpha(t_k-t)}-1|\nonumber\\
\leq&Ct_k^{\frac12}e^{-\alpha t}e^{\alpha(t-t_k)}\tau^{\frac12}+e^{-\alpha t}\alpha\tau\frac{e^{\alpha\tau-1}}{\alpha\tau}\nonumber\\
\leq&C(t+\tau)^{\frac12}e^{-\alpha t}\tau^{\frac12}.\nonumber
\end{align}
We have used the fact $e^{\alpha(t-t_k)}\leq e^{\alpha\tau}\leq e$.
\end{proof}

\begin{rk}
From (\ref{use}), we can also prove that 
$$\|S_{\tau}^k\|_{\mathcal{L}(L^2,L^2)}\leq Ce^{-\alpha t},$$
where $k$ and $t$ satisfying $t\in[t_{k-1},t_{k+1}].$
\end{rk}

Next theorem gives the time-independent weak error of the solutions for different cases.

\begin{tm}\label{result2}
Assume that $u_0\in \dot{H}^2$, $u_N^0=u_N(0)=\pi_Nu_0$ and $\|Q^{\frac12}\|^2_{\mathcal{HS}(L^2,\dot{H}^2)}<\infty$. For the cases $\lambda= 0$ or $-1$, the weak errors are independent of time and of order $\frac12$. 
That is, for any $\phi\in C_b^2(L^2)$, there exists a constant $C=C(u_0,\phi)$ independent of $N,T$ and $M$, such that for any $T=M\tau$,\\
$$\Big{|}E[\phi(u_N(T))]-E[\phi(u_N^M)]\Big{|}\leq C\tau^{\frac12}.$$

\end{tm}
\begin{cor}
Under above assumptions, for any $t\in[(M-1)\tau,(M+1)\tau]$, it also holds
$$\Big{|}E[\phi(u_N(t))]-E[\phi(u_N^M)]\Big{|}\leq C\tau^{\frac12}.$$
\begin{proof}
Let $T=M\tau$. As
\begin{align*}
\Big{|}E[\phi(u_N(t))]-E[\phi(u_N^M)]\Big{|}=\Big{|}E[\phi(u_N(T))]-E[\phi(u_N(t))]\Big{|}+\Big{|}E[\phi(u_N(T))]-E[\phi(u_N^M)]\Big{|}
\end{align*}
and
\begin{align*}
&\Big{|}E[\phi(u_N(T))]-E[\phi(u_N(t))]\Big{|}
\leq\|\phi\|_{C^1_b}E\|u_N(T)-u_N(t)\|_0\\
\leq&\|\phi\|_{C^1_b}(T-t)\sup_{t\geq0}\bigg{[}E\|u_N(t)\|_2+E\|u_N(t)\|_0+E\|u_N(t)\|_1^2\|u_N(t)\|_0\bigg{]}\\
&+\|\phi\|_{C^1_b}E\|\pi_NQ^{\frac12}\big{(}W(T)-W(t)\big{)}\|_0
\leq C\tau^{\frac12},
\end{align*}
we then complete the proof according to Theorem \ref{result2}.
\end{proof}
\end{cor}
\begin{proof}[Proof of Theorem \ref{result2}]
We split it into several steps.

%%%%%%%%%%%%%%%%%
\textbf{Step 1.} Calculation of $E[\phi(u_N(T))]$.

Recall the process we constructed in the proof of Theorem \ref{weakconvergence},
$$dY_N(t)=H_N(Y_N(t))dt+S(T-t)\pi_NQ^{\frac12}dW(t).$$
Now we denote $v_N(T-t,y)=E[\phi(Y_N(T))|Y_N(t)=y]$, then \begin{align}\label{formula}
v_N(0,Y_N(T))=v_N(T,Y_N(0))+\int_0^T\Big{(}Dv_N(T-t,Y_N(t)),S(T-t)\pi_NQ^{\frac12}dW(t)\Big{)},
\end{align}
where
\begin{align}
&v_N(0,Y_N(T))=E[\phi(u_N(T))|Y_N(T)=u_N(T)],\nonumber\\
&v_N(T,Y_N(0))=E[\phi(Y_N(T))|Y_N(0)=S(T)u_N(0)]\nonumber\\
&=E\left[\phi\left(S(T)u_N(0)+\int_0^TH_N(Y_N(s))ds+\int_0^TS(T-s)\pi_NQ^{\frac12}dW\right)\right|Y_N(0)=S(T)u_N(0)\Big{]}.\nonumber
\end{align}
The expectation of (\ref{formula}) implies,
\begin{align}\label{3.1}
E[\phi(u_N(T))]=E\left[\phi\left(S(T)u_N(0)+\int_0^TH_N(Y_N(s))ds+\int_0^TS(T-s)\pi_NQ^{\frac12}dW\right)\right].
\end{align}

%%%%%%%%%%%%%
\textbf{Step 2.} Calculation of $E[\phi(u_N^M)]$.

Similar to \cite{debussche2}, we define a discrete modified process
\begin{align}
Y_N^k:=&S_{\tau}^{M-k}u_N^k\nonumber\\
=&S_{\tau}^Mu_N^0+\mathbf{i}\lambda\tau e^{\alpha\tau}\sum_{l=1}^kS_{\tau}^{M+1-l}\pi_N\left( \frac{|u_N^l|^2+|e^{-\alpha\tau}u_N^{l-1}|^2}2u_N^l\right)+e^{\alpha\tau}\sum_{l=1}^kS_{\tau}^{M+1-l}\pi_NQ^{\frac12}\delta W_l\nonumber\\
=&S_{\tau}^Mu_N^0+\mathbf{i}\lambda\tau e^{\alpha\tau}\sum_{l=1}^kS_{\tau}^{M+1-l}\pi_N\left( \frac{|S_{\tau}^{l-M}Y_N^l|^2+|e^{-\alpha\tau}S_{\tau}^{l-1-M}Y_N^{l-1}|^2}2S_{\tau}^{l-M}Y_N^l\right)\\
&+e^{\alpha\tau}\sum_{l=1}^kS_{\tau}^{M+1-l}\pi_NQ^{\frac12}\delta W_l.\nonumber
\end{align}
Consider the following time continuous interpolation of $Y_N^k$, which is also $V_N$-valued and $\{\mathcal{F}_t\}_{t\geq0}$-adaped,
\begin{align*}
\tilde{Y}_N(t):=&S_{\tau}^Mu_N^0+\mathbf{i}\lambda e^{\alpha\tau}\int_0^t\sum_{l=1}^{M}S_{\tau}^{M+1-l}\pi_N\left(\frac{|S_{\tau}^{l-M}Y_N^l|^2+|e^{-\alpha\tau}S_{\tau}^{l-1-M}Y_N^{l-1}|^2}2S_{\tau}^{l-M}Y_N^l\right)1_l(s)ds\nonumber\\
&+e^{\alpha\tau}\int_0^t\sum_{l=1}^{M}S_{\tau}^{M+1-l}\pi_NQ^{\frac12}1_l(s)dW(s)\\
=:&S_{\tau}^Mu_N^0+\int_0^tH_{\tau}(Y_N^M,s)ds+e^{\alpha\tau}\int_0^t\sum_{l=1}^{M}S_{\tau}^{M+1-l}\pi_NQ^{\frac12}1_l(s)dW(s).
\end{align*}
In particular for $t\in[t_{l-1},t_{l}]$,
\begin{align}\label{y1}
\tilde{Y}_N(t)=&Y_N^{l-1}+\mathbf{i}\lambda e^{\alpha\tau}S_{\tau}^{M+1-l}\pi_N\bigg{(}\frac{|S_{\tau}^{l-M}Y_N^{l}|^2+|e^{-\alpha\tau}S_{\tau}^{l-1-M}Y_N^{l-1}|^2}2S_{\tau}^{l-M}Y_N^{l}\bigg{)}(t-t_{l-1})\nonumber\\
&+e^{\alpha\tau}S_{\tau}^{M+1-l}\pi_NQ^{\frac12}\Big{(}W(t)-W(t_{l-1})\Big{)},\\
\text{or equivalently,}\nonumber\\
\tilde{Y}_N(t)=&Y_N^{l}+\mathbf{i}\lambda e^{\alpha\tau}S_{\tau}^{M+1-l}\pi_N\bigg{(}\frac{|S_{\tau}^{l-M}Y_N^{l}|^2+|e^{-\alpha\tau}S_{\tau}^{l-1-M}Y_N^{l-1}|^2}2S_{\tau}^{l-M}Y_N^{l}\bigg{)}(t-t_{l})\nonumber\\\label{y2}
&+e^{\alpha\tau}S_{\tau}^{M+1-l}\pi_NQ^{\frac12}\Big{(}W(t)-W(t_{l})\Big{)}.
\end{align}
Apply It\^{o}'s formula to $t\mapsto v_N(T-t,\tilde{Y}_N(t))$,
\begin{align}
&dv_N(T-t,\tilde{Y}_N(t))\nonumber\\
=&\frac{\partial v_N}{\partial t}(T-t,\tilde{Y}_N(t))dt
+\Big{(}Dv_N,H_{\tau}(Y_N^M,t)dt+e^{\alpha\tau}\sum_{l=1}^{M}S_{\tau}^{M+1-l}\pi_NQ^{\frac12}1_l(t)dW(t)\Big{)}\nonumber\\
&+\frac12Tr\left[\left(e^{\alpha\tau}\sum_{l=1}^{M}S_{\tau}^{M+1-l}\pi_NQ^{\frac12}1_l(t)\right)^*D^2v_N\left(e^{\alpha\tau}\sum_{l=1}^{M}S_{\tau}^{M+1-l}\pi_NQ^{\frac12}1_l(t)\right)\right]dt\nonumber\\
=&\left(Dv_N,H_{\tau}(Y_N^M,t)-H_N(\tilde{Y}_N(t))\right)dt
+\left(Dv_N,e^{\alpha\tau}\sum_{l=1}^{M}S_{\tau}^{M+1-l}\pi_NQ^{\frac12}1_l(t)dW(t)\right)\nonumber\\
&+\frac12\sum_{l=1}^{M}Tr\left[\left(e^{\alpha\tau}S_{\tau}^{M+1-l}\pi_NQ^{\frac12}\right)^*D^2v_N\left(e^{\alpha\tau}S_{\tau}^{M+1-l}\pi_NQ^{\frac12}\right)\right]1_l(t)dt\nonumber\\
&-\frac12\sum_{l=1}^{M}Tr\left[\left(S(T-t)\pi_NQ^{\frac12}\right)^*D^2v_N\left(S(T-t)\pi_NQ^{\frac12}\right)\right]1_l(t)dt,\nonumber
\end{align}
where $Dv_N$ and $D^2v_N$ are evaluated at $(T-t,\tilde{Y}_N(t))$.

The same as before, integrate the formula above from 0 to T, and take expectation based on the fact that
\begin{align*}
v_N(0,\tilde{Y}_N(T))=&E[\phi(Y_N(T))|Y_N(T)=\tilde{Y}_N(T)]
=E[\phi(u_N^M)|Y_N(T)=u_N^M],\nonumber\\
v_N(T,\tilde{Y}_N(0))=&E[\phi(Y_N(T))|Y_N(0)=\tilde{Y}_N(0)]\nonumber\\
=&E\left[\phi\left(S_{\tau}^Mu_N(0)+\int_0^TH_N(Y_N(s))ds+\int_0^TS(T-s)\pi_NQ^{\frac12}dW\right)\bigg{|}Y_N(0)=S_{\tau}^Mu_N(0)\right],\nonumber
\end{align*}
we get
\begin{align}\label{3.2}
E[\phi(u_N^M)]=&E\left[\phi\left(S_{\tau}^Mu_N(0)+\int_0^TH_N(Y_N(s))ds+\int_0^TS(T-s)\pi_NQ^{\frac12}dW\right)\right]\nonumber\\
+&E\int_0^T\left(Dv_N,H_{\tau}(Y_N^M,t)-H_N(\tilde{Y}_N(t))\right)dt\nonumber\\
+&\frac12\sum_{l=1}^{M}E\int_0^TTr\bigg{[}\left(e^{\alpha\tau}S_{\tau}^{M+1-l}\pi_NQ^{\frac12}\right)^*D^2v_N\left(e^{\alpha\tau}S_{\tau}^{M+1-l}\pi_NQ^{\frac12}\right)\nonumber\\
-&\left(S(T-t)\pi_NQ^{\frac12}\right)^*D^2v_N\left(S(T-t)\pi_NQ^{\frac12}\right)\bigg{]}1_l(t)dt.
\end{align}

%%%%%%%%%%%%
\textbf{Step 3.} Weak convergence order.

Subtracting (\ref{3.1}) from (\ref{3.2}), we derive
\begin{align}
&E[\phi(u_N^M)]-E[\phi(u_N(T))]\nonumber\\
=&E\bigg{[}\phi\left(S_{\tau}^Mu_N(0)+\int_0^TH_N(Y_N(s))ds+\int_0^TS(T-s)\pi_NQ^{\frac12}dW\right)\nonumber\\
&-\phi\left(S(T)u_N(0)+\int_0^TH_N(Y_N(s))ds+\int_0^TS(T-s)\pi_NQ^{\frac12}dW\right)\bigg{]}\nonumber\\
&+E\int_0^T\left(Dv_N,H_{\tau}(Y_N^M,t)-H_N(\tilde{Y}_N(t))\right)dt\nonumber\\
&+\frac12\sum_{l=1}^{M}E\int_0^TTr\bigg{[}\left(e^{\alpha\tau}S_{\tau}^{M+1-l}\pi_NQ^{\frac12}\right)^*D^2v_N\left(e^{\alpha\tau}S_{\tau}^{M+1-l}\pi_NQ^{\frac12}\right)\nonumber\\
&-\left(S(T-t)\pi_NQ^{\frac12}\right)^*D^2v_N\left(S(T-t)\pi_NQ^{\frac12}\right)\bigg{]}1_l(t)dt.\nonumber\\
=:&~\uppercase\expandafter{\romannumeral1}+\uppercase\expandafter{\romannumeral2}+\uppercase\expandafter{\romannumeral3}.\nonumber
\end{align}
Now we estimate $\uppercase\expandafter{\romannumeral1}$, $\uppercase\expandafter{\romannumeral2}$, and $\uppercase\expandafter{\romannumeral3}$ separately. The constants C below may be different but are all independent of T and $\tau$.

\begin{align}\label{31}
|\uppercase\expandafter{\romannumeral1}|=&\bigg{|}E\left[\phi\left(S_{\tau}^Mu_N(0)+\int_0^TH_N(Y_N(s))ds+\int_0^TS(T-s)\pi_NQ^{\frac12}dW\right)\right]\nonumber\\
&-E\left[\phi\left(S(T)u_N(0)+\int_0^TH_N(Y_N(s))ds+\int_0^TS(T-s)\pi_NQ^{\frac12}dW\right)\right]\bigg{|}\nonumber\\
\leq&C\|\phi\|_{C_b^1}\|S_{\tau}^Mu_N(0)-S(T)u_N(0)\|_0\nonumber\\
\leq&C\|\phi\|_{C_b^1}\|S_{\tau}^M-S(T)\|_{\mathcal{L}(\dot{H}^2,L^2)}\|u_N(0)\|_2\nonumber\\
\leq&C(T+\tau)^{\frac12}e^{-\alpha T}\tau^{\frac12},
\end{align}
where we have used Lemma \ref{operator2} and $u_N(0)=\pi_Nu_0\in \dot{H}^2$.

Noticing $\uppercase\expandafter{\romannumeral2}=0$ for $\lambda=0$, now we consider the nonlinear term $\uppercase\expandafter{\romannumeral2}$ for $\lambda=-1$. 
By using the notation $a_l:=S_{\tau}^{l-M}Y_N^l=u_N^l$ and (\ref{y1}) and (\ref{y2}), we can define $b_l$ in two ways,
\begin{align*}
b_l:=&S(t-T)\tilde{Y}_N(t)1_l(t)\nonumber\\
=&S(t-T)S_{\tau}^{M+1-l}u_N^{l-1}+e^{\alpha\tau}S(t-T)S_{\tau}^{M+1-l}\bigg{(}\mathbf{i}\lambda\pi_N\left(\frac{|e^{-\alpha\tau}u_N^{l-1}|^2+|u_N^l|^2}2u_N^l\right)(t-t_{l-1})\\
&+\pi_NQ^{\frac12}\left(W(t)-W(t_{l-1})\right)\bigg{)},\end{align*}
or equivalently,
\begin{align*}
b_l:=&S(t-T)\tilde{Y}_N(t)1_l(t)\nonumber\\
=&S(t-T)S_{\tau}^{M-l}u_N^l+e^{\alpha\tau}S(t-T)S_{\tau}^{M+1-l}\bigg{(}\mathbf{i}\lambda\pi_N\left(\frac{|e^{-\alpha\tau}u_N^{l-1}|^2+|u_N^l|^2}2u_N^l\right)(t-t_l)\\
&+\pi_NQ^{\frac12}\left(W(t)-W(t_l)\right)\bigg{)}.
\end{align*}
Hence, we have
\begin{align*}
&a_{l-1}-b_l\\=&\left(Id-S(t-T)S_{\tau}^{M+1-l}\right)u_N^{l-1}\\
&-e^{\alpha\tau}S(t-T)S_{\tau}^{M+1-l}\left(\mathbf{i}\lambda\pi_N\left(\frac{|e^{-\alpha\tau}u_N^{l-1}|^2+|u_N^l|^2}2u_N^l\right)(t-t_{l-1})+\pi_NQ^{\frac12}\left(W(t)-W(t_{l-1})\right)\right)
\end{align*}
and
\begin{align*}
a_l-b_l=&\left(Id-S(t-T)S_{\tau}^{M-l}\right)u_N^l\\
&-e^{\alpha\tau}S(t-T)S_{\tau}^{M+1-l}\left(\mathbf{i}\lambda\pi_N\left(\frac{|e^{-\alpha\tau}u_N^{l-1}|^2+|u_N^l|^2}2u_N^l\right)(t-t_l)+\pi_NQ^{\frac12}\left(W(t)-W(t_l)\right)\right),
\end{align*}
where $\|S(t-T)S_{\tau}^{M+1-l}\|_{\mathcal{L}(L^2,L^2)}\leq C$ and
\begin{align*}
\|Id-S(t-T)S_{\tau}^{M-l}\|_{\mathcal{L}(\dot{H}^2,L^2)}\leq\|S(t-T)\|_{\mathcal{L}(L^2,L^2)}\|S(T-t)-S_{\tau}^{M-l}\|_{\mathcal{L}(\dot{H}^2,L^2)}\leq C(T-t+\tau)^{\frac12}\tau^{\frac12}
\end{align*}
 according to Lemma \ref{operator2}. Thus, we have the following estimate
 \begin{align*}
 \|a_l-b_l\|_0\leq& C\Big{[}(T-t+\tau)^{\frac12}\tau^{\frac12}\|u_N^l\|_2
 +\tau\Big{(}\|u_N^{l-1}\|_1^2+\|u_N^l\|_1^2\Big{)}\|u_N^l\|_0+\|\pi_NQ^{\frac12}(W(t)-W(t_l))\|_0\Big{]}.
 \end{align*}
 Also,  $\|a_{l-1}-b_l\|_0$ can be estimated in the same way. 
Thus, based on \eqref{dv}, we have
\begin{align}\label{begin}
|\uppercase\expandafter{\romannumeral2}|=&\left|E\int_0^T\left(Dv_N,H_{\tau}(Y_N^M,t)-H_N(\tilde{Y}_N(t))\right)dt\right|
\leq C\|\phi\|_{C_b^1}\int_0^TE\|H_{\tau}(Y_N^M,t)-H_N(\tilde{Y}_N(t))\|_0dt,
\end{align}
where
\begin{align*}
&H_{\tau}(Y_N^M,t)-H_N(\tilde{Y}_N(t))\nonumber\\
=&\sum_{l=1}^{M}\left[e^{\alpha\tau}S_{\tau}^{M+1-l}\pi_N\left(\mathbf{i}\lambda\frac{|e^{-\alpha\tau}a_{l-1}|^2+|a_l|^2}2a_l\right)-S(T-t)\pi_N\left(\mathbf{i}\lambda|b_l|^2b_l\right)\right]1_l(t)\nonumber\\
=&\frac{\lambda}2\mathbf{i}\sum_{l=1}^{M}\bigg{[}e^{\alpha\tau}\Big{(}S_{\tau}^{M+1-l}-S(T-t)\Big{)}\pi_N\left(|e^{-\alpha\tau}a_{l-1}|^2a_l\right)+(e^{-\alpha\tau}-1)S(T-t)\pi_N\left(|a_{l-1}|^2a_l\right)\\
&+S(T-t)\pi_N\left(|a_{l-1}|^2a_l-|b_l|^2b_l\right)\bigg{]}1_l(t)\\
&+\frac{\lambda}2\mathbf{i}\sum_{l=1}^{M}\bigg{[}e^{\alpha\tau}\Big{(}S_{\tau}^{M+1-l}-S(T-t)\Big{)}\pi_N\left(|a_l|^2a_l\right)+(e^{\alpha\tau}-1)S(T-t)\pi_N\left(|a_{l}|^2a_l\right)\\
&+S(T-t)\pi_N\left(|a_l|^2a_l-|b_l|^2b_l\right)\bigg{]}1_l(t)\\
=&\frac{\lambda}2\mathbf{i}\bigg{[}\sum_{l=1}^{M}e^{\alpha\tau}\Big{(}S_{\tau}^{M+1-l}-S(T-t)\Big{)}\pi_N\left(|e^{-\alpha\tau}a_{l-1}|^2a_l\right)1_l(t)+\sum_{l=1}^{M}S(T-t)\pi_N\Big{(}|a_{l-1}|^2\left(a_l-b_l\right)\Big{)1_l(t)}\\
&+\sum_{l=1}^{M}S(T-t)\pi_N\Big{(}|b_l|^2(a_{l-1}-b_l)\Big{)}1_l(t)+\sum_{l=1}^{M}S(T-t)\pi_N\Big{(}a_{l-1}b_l(\overline{a}_{l-1}-\overline{b}_l)\Big{)}1_l(t)\\
&+\sum_{l=1}^{M}(e^{-\alpha\tau}-1)S(T-t)\pi_N\Big{(}|a_{l-1}|^2a_l\Big{)1_l(t)}+\sum_{l=1}^{M}e^{\alpha\tau}\Big{(}S_{\tau}^{M+1-l}-S(T-t)\Big{)}\pi_N\left(|a_l|^2a_l\right)1_l(t)\\
&+\sum_{l=1}^{M}S(T-t)\pi_N\Big{(}|a_l|^2\left(a_l-b_l\right)\Big{)}1_l(t)+\sum_{l=1}^{M}S(T-t)\pi_N\Big{(}|b_l|^2(a_l-b_l)\Big{)}1_l(t)\\
&+\sum_{l=1}^{M}S(T-t)\pi_N\Big{(}a_lb_l(\overline{a}_l-\overline{b}_l)\Big{)}1_l(t)\bigg{]}+\sum_{l=1}^{M}(e^{\alpha\tau}-1)S(T-t)\pi_N\Big{(}|a_{l}|^2a_l\Big{)1_l(t)}\nonumber\\
:=&\frac{\lambda}{2}\mathbf{i}\Big{[}\uppercase\expandafter{\romannumeral2}_1^{l-1}+\uppercase\expandafter{\romannumeral2}_2^{l-1}+\uppercase\expandafter{\romannumeral2}_3^{l-1}+\uppercase\expandafter{\romannumeral2}_4^{l-1}+\uppercase\expandafter{\romannumeral2}_5^{l-1}+\uppercase\expandafter{\romannumeral2}_1^{l}+\uppercase\expandafter{\romannumeral2}_2^l+\uppercase\expandafter{\romannumeral2}_3^l+\uppercase\expandafter{\romannumeral2}_4^l+\uppercase\expandafter{\romannumeral2}_5^l\Big{]}.
\end{align*}

If $\lambda=-1$, thanks to the uniform estimations of 0-norm, 1-norm and 2-norm of $u_N^k$, we have the following estimates.

By the embedding $H^1\hookrightarrow L^{\infty}$ in $\mathbb{R}^1$, we have following exponential estimates
\begin{align*}
E\|\uppercase\expandafter{\romannumeral2}_1^{l-1}\|_0\leq&\frac12\sum_{l=1}^{M}\|S_{\tau}^{M+1-l}-S(T-t)\|_{\mathcal{L}(\dot{H}^2,L^2)}E\Big{\|}\pi_N\left(|e^{-\alpha\tau}u_N^{l-1}|^2u_N^l\right)\Big{\|}_21_l(t)\\
\leq&C\sum_{l=1}^{M}\|S_{\tau}^{M+1-l}-S(T-t)\|_{\mathcal{L}(\dot{H}^2,L^2)}E\left[\|u_N^{l-1}\|_1^4+\|u_N^l\|_2^2\right]1_l(t)\\
\leq&C(T-t+\tau)^{\frac12}e^{-\alpha(T-t)}\tau^{\frac12},
\end{align*}
\begin{align*}
E\|\uppercase\expandafter{\romannumeral2}_2^{l-1}\|_0\leq&Ce^{-\alpha(T-t)}E\sum_{l=1}^{M}\|a_{l-1}\|_1^2\|a_l-b_l\|_01_l(t)\\
\leq &Ce^{-\alpha(T-t)}E\sum_{l=1}^{M}\|u_N^{l-1}\|_1^2\Bigg{[}C(T-t+\tau)^{\frac12}\tau^{\frac12}\|u_N^l\|_2\\
&+C\left[\Big{(}\|u_N^{l-1}\|_1^2+\|u_N^l\|_1^2\Big{)}\|u_N^l\|_0\tau+\|\pi_NQ^{\frac12}(W(t)-W(t_l))\|_0\right]\Bigg{]}1_l(t)\\
\leq&C(T-t+1)^{\frac12}e^{-\alpha(T-t)}\tau^{\frac12},\\
E\|\uppercase\expandafter{\romannumeral2}_5^{l-1}\|_0\leq& e^{-\alpha(T-t)}(1-e^{-\alpha\tau})E\bigg{[}\|u_N^{l-1}\|_1^2\|u_N^l\|_0\bigg{]}\leq Ce^{-\alpha(T-t)}\tau,
\end{align*}
and their integrals are also of order $\frac12$. 
$\uppercase\expandafter{\romannumeral2}_1^l$, $\uppercase\expandafter{\romannumeral2}_2^l$  and $\uppercase\expandafter{\romannumeral2}_5^l$ can also be estimated in the same way, where we have used the fact that for any $T>0$, the integral $\int_0^T(T-t+\tau)^{\frac12}e^{-\alpha(T-t)}dt$ is bounded and $\sum_{l=1}^{M}1_l(t)=1$.

Other terms are proved in the same procedure by using the fact that
\begin{align*}
\|b_l\|_{L^{\infty}}^2\leq C\|S(t-T)S_{\tau}^{M-l}\|_{\mathcal{L}(\dot{H}^1,\dot{H}^1)}^2[\|u_N^l\|_1^4+\|u_N^{l-1}\|_1^4+\|\pi_NQ^{\frac12}\delta W_l\|_1^2]
\end{align*}
and
\begin{align*}
\|a_lb_l\|_{L^{\infty}}\leq\frac12[\|a_l\|_{L^{\infty}}^2+\|b_l\|_{L^{\infty}}^2].
\end{align*}
Finally, we have
\begin{align}\label{32}
|\uppercase\expandafter{\romannumeral2}|\leq&C\tau^{\frac12}.
\end{align}
Next is the estimate of $\uppercase\expandafter{\romannumeral3}$, which is similar to the same part in the proof of Theorem \ref{weakconvergence}.
\begin{align*}
\uppercase\expandafter{\romannumeral3}=&\frac12\sum_{l=1}^{M}E\int_0^TTr\bigg{[}\left(e^{\alpha\tau}S_{\tau}^{M+1-l}\pi_NQ^{\frac12}\right)^*D^2v_N\left(e^{\alpha\tau}S_{\tau}^{M+1-l}\pi_NQ^{\frac12}\right)\nonumber\\
&-\left(S(T-t)\pi_NQ^{\frac12}\right)^*D^2v_N\left(S(T-t)\pi_NQ^{\frac12}\right)\bigg{]}1_l(t)dt\nonumber\\
=&\frac12\sum_{l=1}^{M}E\int_0^TTr\bigg{[}\left(\left(e^{\alpha\tau}S_{\tau}^{M+1-l}-S(T-t)\right)\pi_NQ^{\frac12}\right)^*D^2v_N\left(\left(e^{\alpha\tau}S_{\tau}^{M+1-l}-S(T-t)\right)\pi_NQ^{\frac12}\right)\bigg{]}\nonumber\\
&+2Tr\bigg{[}\left(\left(e^{\alpha\tau}S_{\tau}^{M+1-l}-S(T-t)\right)\pi_NQ^{\frac12}\right)^*D^2v_N\left(S(T-t)\pi_NQ^{\frac12}\right)\bigg{]}1_l(t)dt\nonumber\\
=&\frac12\sum_{l=1}^ME\int_0^TTr\Bigg{[}
e^{2\alpha\tau}\left(\left(S_{\tau}^{M+1-l}-S(T-t)\right)\pi_NQ^{\frac12}\right)^*D^2v_N\left(\left(S_{\tau}^{M+1-l}-S(T-t)\right)\pi_NQ^{\frac12}\right)\\
&+2e^{2\alpha\tau}\left(\left(S_{\tau}^{M+1-l}-S(T-t)\right)\pi_NQ^{\frac12}\right)^*D^2v_N\left(S(T-t)\pi_NQ^{\frac12}\right)\\
&+(e^{2\alpha\tau}-1)\left(S(T-t)\pi_NQ^{\frac12}\right)^*D^2v_N\left(S(T-t)\pi_NQ^{\frac12}\right)
\Bigg{]}1_l(t)dt\\
:=&\frac12\sum_{l=1}^{M}E\int_0^T(A_l+2B_l+C_l)1_l(t)dt,
\end{align*}
where $A_l$, $B_l$ and $C_l$ satisfy
\begin{align*}
E|A_l|\leq& C\|S_{\tau}^{M+1-l}-S(T-t))\|_{\mathcal{L}(\dot{H}^2,L^2)}^2\|\pi_NQ^{\frac12}\|^2_{\mathcal{L}(L^2,\dot{H}^2)}\|\phi\|_{C_b^2}\leq C(T-t+\tau)e^{-2\alpha(T-t)}\tau,\\
E|B_l|\leq&C\|S_{\tau}^{M+1-l}-S(T-t))\|_{\mathcal{L}(\dot{H}^2,L^2)}\|\pi_NQ^{\frac12}\|^2_{\mathcal{L}(L^2,\dot{H}^2)}\|\phi\|_{C_b^2}\|S(T-t)\|_{\mathcal{L}(L^2,L^2)}\nonumber\\
\leq&C(T-t+\tau)^{\frac12}e^{-2\alpha(T-t)}\tau^{\frac12}\nonumber
\end{align*}
and
\begin{align*}
E|C_l|\leq C\tau\|\pi_NQ^{\frac12}\|^2_{\mathcal{L}(L^2,L^2)}\|\phi\|_{C_b^2}\|S(T-t)\|_{\mathcal{L}(L^2,L^2)}^2\leq Ce^{-2\alpha(T-t)}\tau.
\end{align*}
It follows
\begin{align}\label{33}
|\uppercase\expandafter{\romannumeral3}|\leq C\tau^{\frac12}.
\end{align}
We can conclude from (\ref{31}), (\ref{32}) and (\ref{33}) that, 
\begin{align*}
\Big{|}E\left[\phi(u_N(T))\right]-E\left[\phi(u_N^M)\right]\Big{|}
\leq C\tau^{\frac12},
\end{align*}
where $C$ is independent of $T,~M$ and $N$.

\end{proof}

%%%%%%%%%%%%%%%%%%%%%%%
\subsection{Convergence order between invariant measures $\mu_N$ and $\mu_N^{\tau}$}

\begin{tm}
For $\lambda=0$ or $-1$, assume that $u_0\in\dot{H}^2$ and $\|Q^{\frac12}\|_{\mathcal{HS}(L^2,\dot{H}^2)}<\infty$, the error between invariant measures $\mu_N$ and $\mu_N^{\tau}$ is of order $\frac12$, i.e.,
\begin{align*}
\left|\int_{V_N}\phi(y)d\mu_N(y)-\int_{V_N}\phi(y)d\mu^{\tau}_N(y)\right|<C\tau^{\frac12},\;\;\forall\;\phi\in C_b^2(L^2).
\end{align*}
\end{tm}
\begin{proof}
By the ergodicity of stochastic processes $u_N$ and $u_N^k$, we have
\begin{align}
\lim_{T\to\infty}\frac{1}{T}\int_0^TE\phi\big{(}u_N(t)\big{)}dt=&\int_{V_N}\phi(y)d\mu_N(y),\\
\lim_{M\to\infty}\frac{1}{M}\sum_{k=0}^{M-1}E\phi(u_N^k)=&\int_{V_N}\phi(y)d\mu_N^{\tau}(y)
\end{align}
for any $\phi\in C_b^2(L^2)$.
As the weak error is proved to be independent of step $k$ and time $t$ in Theorem \ref{result2}, it turns out that for a fixed $\tau$,

\begin{align*}
&\left|\int_{V_N}\phi(y)d\mu_N(y)-\int_{V_N}\phi(y)d\mu^{\tau}_N(y)\right|\\
\le&\lim_{\substack{M\to\infty,\\T=M\tau\to\infty}}\frac{1}{T}\sum_{k=0}^{M-1}\int_{t_k}^{t_{k+1}}\left|E\phi\big{(}u_N(t)\big{)}-E\phi(u_N^k)\right|dt
\leq C\tau^{\frac12}.
\end{align*}
\iffalse
\begin{align}\label{numergodic}
&\left|\frac{1}{T}\int_0^TE\phi\big{(}u_N(t)\big{)}dt-\int_{V_N}\phi(y)d\mu_N(y)\right|\leq C\tau^{\frac12},\\\label{numergodic2}
&\left|\frac{1}{M}\sum_{k=0}^{M-1}E\phi(u_N^k)-\int_{V_N}\phi(y)d\mu^{\tau}_N(y)\right|\leq C\tau^{\frac12}.
\end{align}
Using  (\ref{numergodic}), (\ref{numergodic2}) and Theorem \ref{result2}, we have the following estimate
\begin{align*}
&\left|\int_{V_N}\phi(y)d\mu_N(y)-\int_{V_N}\phi(y)d\mu^{\tau}_N(y)\right|\nonumber
\leq\left|\frac{1}{T}\int_0^TE\phi\big{(}u_N(t)\big{)}dt-\int_{V_N}\phi(y)d\mu_N(y)\right|\\
+&\left|\frac{1}{M}\sum_{k=0}^{M-1}E\phi(u_N^k)-\int_{V_N}\phi(y)d\mu^{\tau}_N(y)\right|
+\left|\frac{1}{T}\sum_{k=0}^{M-1}\int_{t_k}^{t_{k+1}}E\phi\big{(}u_N(t)\big{)}-E\phi(u_N^k)dt\right|
\leq C\tau^{\frac12},
\end{align*}
where the third term takes advantage of the independence of time $t$ and step $k$ for the weak error $E\phi\big{(}u_N(t)\big{)}-E\phi(u_N^k)$.
\fi
\end{proof}
\begin{rk}
For the case $\lambda=1$, if the 1-norm and 2-norm of $u_N^k$ is also uniformly bounded, we can also get order $\frac12$ for both time-independent weak error and error between invariant measures. If not, based on the fact $\|\cdot\|_{s+1}\leq N\|\cdot\|_s$, we can get the weak error depend on $N$
$$\Big{|}E[\phi(u_N(T))]-E[\phi(u_N^M)]\Big{|}\leq CN^4\tau^{\frac12},$$
as well as the error between invariant measures. 
\end{rk}

\section{\textsc{\Large{A}ppendix}}

\subsection{The proof of proposition \ref{1norm}}\label{5.1}

$\romannumeral1)$ 
As it is proved in Part 3 of Theorem \ref{spaceergodic} that $E\|u_N(t)\|_0^2<C$, we assume further that $E\|u_N(t)\|_0^{2n}<C,\;\forall\;n=1,\cdots,p-1$. Denoting $dM_1:=2Re\left({u}_N,\pi_NQ^{\frac{1}{2}}dW\right)$, then It\^o's formula and \eqref{un} yields
\begin{align*}
d\|u_N(t)\|_0^{2p}=&p\|u_N(t)\|_0^{2(p-1)}d\|u_N(t)\|_0^2
+\frac12p(p-1)\|u_N(t)\|_0^{2(p-2)}d\langle M_1\rangle\\
\le&-2\alpha p\|u_N(t)\|_0^{2p}dt+p\|u_N(t)\|_0^{2(p-1)}dM_1(t)+2p(2p-1)\sum_{m=1}^N\eta_m\|u_N(t)\|_0^{2(p-1)}dt,
\end{align*}
where $\langle\cdot\rangle$ denotes the quadratic variation process and in the last step we used the fact
\begin{align*}
d\langle M_1\rangle=&4\left\langle Re\sum_{m=1}^N\int_0^1\overline{u}_N(s)\sqrt{\eta_m}e_m(x)dx(d\beta_{m,1}+\mathbf{i}d\beta_{m,2})\right\rangle\\
=&4\sum_{m=1}^N\left[\left(Re\int_0^1\overline{u}_N(t,x)\sqrt{\eta_m}e_m(x)dx\right)^2+\left(Im\int_0^1\overline{u}_N(t,x)\sqrt{\eta_m}e_m(x)dx\right)^2\right]dt\\
\le&8\sum_{m=1}^N\eta_m\|u_N(t)\|_0^2dt.
\end{align*}
Taking expectation on both sides of above equation, we obtain
\begin{align*}
\frac{d}{dt}E\|u_N(t)\|_0^{2p}
\le&-2\alpha pE\|u_N(t)\|_0^{2p}+2p(2p-1)\sum_{m=1}^N\eta_mE\|u_N(t)\|_0^{2(p-1)}\\
\le&-2\alpha pE\|u_N(t)\|_0^{2p}+C
\end{align*}
by induction. 
Then multiplying $e^{2\alpha pt}$ to both sides of above equation yields the result.

$\romannumeral2)$
The proof in this part is similar to the proof of Lemma 2.5 in \cite{debussche}.
According to the Gagliardo-Nirenberg interpolation inequality, there exists a positive constant $c_0$, such that
\begin{align}\label{5.1}
\frac58\lambda\|u_N(t)\|_{L^4}^4\leq\|u_N(t)\|_{L^4}^4\leq\frac14\|\nabla u_N(t)\|_0^2+\frac12c_0\|u_N(t)\|_0^6.
\end{align}
Thus,
\begin{align}
0\leq\mathcal{H}(u_N(t)):=&\frac12\|\nabla u_N(t)\|_0^2-\frac{\lambda}4\|u_N(t)\|_{L^4}^4+c_0\|u_N(t)\|_0^6\nonumber\\\label{5.2}
\leq&\frac23\left(\|\nabla u_N(t)\|_0^2-\lambda\|u_N(t)\|_{L^4}^4+2c_0\|u_N(t)\|_0^6\right).
\end{align}
Applying It\^o's formula to $\mathcal{H}(u_N(t))$, it leads to
\begin{align*}
d\mathcal{H}(u_N(t))=&\bigg{[}-\alpha\|\nabla u_N(t)\|_0^2+\alpha\lambda\|u_N(t)\|_{L^4}^4-6\alpha c_0\|u_N(t)\|_0^6-2\lambda\int_0^1|u_N|^2\sum_{m=1}^N\eta_m|e_m|^2dx
\\
&+\sum_{m=1}^Nm^2\eta_m+6c_0\|u_N(t)\|_0^4\sum_{m=1}^N\eta_m
+12c_0\|u_N(t)\|_0^2\|\pi_NQ^{\frac12}u_N(t)\|_0^2\bigg{]}dt\\
&+6c_0\|u_N(t)\|_0^4Re\left(u_N,\pi_NQ^{\frac12}dW\right)-Re\left(\Delta u_N(t)+\lambda|u_N(t)|^2u_N(t),\pi_NQ^{\frac12}dW\right),
\end{align*}
where we have used the fact $\left((Id-\pi_N)v,v_N\right)=0,\,\forall\,v\in \dot{H}^0,\,v_N\in V_N$.
By the following estimates
\begin{align*}
&-2\lambda\int_0^1|u_N|^2\sum_{m=1}^N\eta_m|e_m|^2dx\leq0,\\
&6c_0\|u_N(t)\|_0^4\sum_{m=1}^N\eta_m+12c_0\|u_N(t)\|_0^2\|\pi_NQ^{\frac12}u_N(t)\|_0^2
\leq4\alpha c_0\|u_N(t)\|_0^6+C
\end{align*}
and \eqref{5.2}, we have
\begin{align}
&d\mathcal{H}(u_N(t))\le\bigg{[}-\alpha\|\nabla u_N(t)\|_0^2+\alpha\lambda\|u_N(t)\|_{L^4}^4-2\alpha c_0\|u_N(t)\|_0^6+\sum_{m=1}^Nm^2\eta_m+C\bigg{]}dt\nonumber\\
&+6c_0\|u_N(t)\|_0^4Re\left(u_N(t),\pi_NQ^{\frac12}dW(t)\right)-Re\left(\Delta u_N(t)+\lambda|u_N(t)|^2u_N(t),\pi_NQ^{\frac12}dW\right)\nonumber\\\label{5.3}
\le&-\frac32\alpha\mathcal{H}(u_N(t))dt+Cdt+dM_2,
\end{align}
where 
$$dM_2:=6c_0\|u_N\|_0^4Re\left(u_N,\pi_NQ^{\frac12}dW\right)-Re\left(\Delta u_N+\lambda|u_N|^2u_N,\pi_NQ^{\frac12}dW\right).$$
 Taking expectation, we derive

\begin{align*}
dE\mathcal{H}(u_N(t))
\leq&-\frac32\alpha E\mathcal{H}(u_N(t))dt+Cdt.
\end{align*}
Hence, by multiplying $e^{\frac32\alpha t}$ to both sides of the equation above and then taking integral from $0$ to $t$, we get the uniform boundedness for $p=1$.
By induction, we assume that the results hold for $p-1$.
Then, based on the following estimates (see \cite{debussche})
\begin{align*}
\left<6\|u_N\|_0^4Re\left(u_N,\pi_NQ^{\frac12}dW\right)\right>^2\le&C\|Q^{\frac12}\|_{\mathcal{HS}(L^2,L^2)}^2\|u_N\|_0^{10}dt,\\
\left<Re\left(\Delta u_N+\lambda|u_N|^2u_N,\pi_NQ^{\frac12}dW\right)\right>^2\le&C\|Q^{\frac12}\|_{\mathcal{HS}(L^2,\dot{H}^1)}^2\left(\|\nabla u_N\|_0^2+\|u_N\|_0^{10}\right)dt
\end{align*}
and \eqref{5.3}, we have
\begin{align}
d\mathcal{H}(u_N(t))^p=&p\mathcal{H}(u_N(t))^{p-1}d\mathcal{H}(u_N(t))+\frac12p(p-1)\mathcal{H}(u_N(t))^{p-2}d\langle M_2\rangle\nonumber\\
\le&-\frac32\alpha p\mathcal{H}(u_N(t))^pdt+Cp\mathcal{H}(u_N(t))^{p-1}dt+p\mathcal{H}(u_N(t))^{p-1}dM_2\nonumber\\\label{5.4}
&+Cp(p-1)\mathcal{H}(u_N(t))^{p-2}\left(\|\nabla u_N(t)\|_0^2+\|u_N(t)\|_0^{10}\right)dt.
\end{align}
From \eqref{5.1}, we deduce that
\begin{equation*}
\mathcal{H}(u_N(t))\ge
\left\{
\begin{aligned}
&\frac12\|\nabla u_N(t)\|_0^2+c_0\|u_N(t)\|_0^6,\;\lambda=0\;\text{or}\,-1,\\
&\frac7{16}\|\nabla u_N(t)\|_0^2+\frac78c_0\|u_N(t)\|_0^6,\;\lambda=1.
\end{aligned}
\right.
\end{equation*}
As a result, the last term in \eqref{5.4} can be estimated as
\begin{align}
&Cp(p-1)\mathcal{H}(u_N(t))^{p-2}\left(\|\nabla u_N(t)\|_0^2+\|u_N(t)\|_0^{10}\right)\nonumber\\\label{5.5}
\le& \left(C\mathcal{H}(u_N(t))+C\mathcal{H}(u_N(t))^{\frac53}\right)\mathcal{H}(u_N(t))^{p-2}
\le C\mathcal{H}(u_N(t))^{p-1}+\frac12\alpha p\mathcal{H}(u_N(t))^p,
\end{align}
where in the last step we used the inequality of arithmetic and geometric means
\begin{align*}
C(\mathcal{H}(u_N(t))^2\cdot\mathcal{H}(u_N(t))^2\cdot\mathcal{H}(u_N(t)))^{\frac13}
\le\frac{\frac34\alpha p\mathcal{H}(u_N(t))^2+\frac34\alpha p\mathcal{H}(u_N(t))^2+C\mathcal{H}(u_N(t))}{3}.
\end{align*}
Gethering \eqref{5.4} and \eqref{5.5} and taking expectation, we obtain
\begin{align*}
dE\mathcal{H}(u_N(t))^p\le-\alpha pE\mathcal{H}(u_N(t))^pdt+Cdt
\end{align*}
by induction, which complete the proof by multiplying $e^{\alpha pt}$ on both sides of above equation.

$\romannumeral3)$ 
We define a functional
\begin{align*}
f(u)=\int_0^1|\Delta u|^2dx+\lambda Re\int_0^1(\Delta\overline{u})|u|^2udx,
\end{align*}
which satisfies 
\begin{align}\label{5.6}
\|\Delta u\|_0^2\le2f(u)+C\|u\|_1^6
\end{align}
based on the continuous embedding $H^1\hookrightarrow L^6$ and
%\begin{align*}
$\left|\lambda Re\int_0^1\Delta\overline{u}|u|^2udx\right|
\le\frac12\|\Delta u\|_0^2+\frac12\|u\|_{L^6}^6
\le\frac12\|\Delta u\|_0^2+C\|u\|_{1}^6.
$ %\end{align*}
The It\^o's formula applied to $f(u_N)$ yields
\begin{align}
df(u_N)=&Df(u_N)\Big{(}\left(\bi\Delta u_N+\bi\lambda|u_N|^2u_N-\alpha u_N\right)dt\Big{)}
+Df(u_N)\Big{(}\pi_NQ^{\frac12}dW\Big{)}\nonumber\\
&+\frac12D^2f(u_N)(\pi_NQ^{\frac12}dW,\pi_NQ^{\frac12}dW)\nonumber\\\label{2norm}
=:&\mathcal{A}+\mathcal{B}+\mathcal{C},
\end{align}
where
\begin{align*}
Df(u)(\varphi)
=&Re\int_0^1\Big{[}2\Delta\overline{u}\Delta\varphi
+2\lambda(\Delta\overline{u})u Re(\overline{u}\varphi)
+\lambda(\Delta\overline{u})|u|^2\varphi
+\lambda(\Delta(|u|^2u))\overline{\varphi}\Big{]}dx,\\
D^2f(u)(\varphi,\psi)
=&Re\int_0^1\Big{[}2\Delta\overline{\varphi}\Delta\psi+2\lambda(\Delta\overline{u})u Re(\overline{\varphi}\psi)
+2\lambda(\Delta\overline{u})\varphi Re(\overline{u}\psi)
+2\lambda(\Delta\overline{\varphi})u Re(\overline{u}\psi)\\
&+2\lambda(\Delta\overline{u})\psi Re(\overline{\varphi}u)
+2\lambda(\Delta\overline{\psi})u Re(\overline{u}\varphi)
+\lambda(\Delta\overline{\varphi})|u|^2\psi
+\lambda(\Delta\overline{\psi})|u|^2\varphi\Big{]}dx
\end{align*}
and $E[\mathcal{B}]=0$. Now we estimate $\mathcal{A}$ and $\mathcal{C}$ respectively.
\begin{align*}
 E[\mathcal{A}]=&-2\alpha E[f(u_N)]dt
+Re E\int_0^1\Big{[}4\lambda\bi(\Delta\overline{u}_N)u_N|\nabla u_N|^2
+2\lambda\bi(\Delta\overline{u}_N)\overline{u}_N(\nabla u_N)^2\Big{]}dxdt\\
&+Re E\int_0^1\Big{[}\lambda^2\bi(\Delta\overline{u}_N)|u_N|^4
-4\alpha\lambda(\Delta\overline{u}_N)u_N|u_N|^2\Big{]}dxdt\\
&+Re E\int_0^1\Big{[}-4\alpha\lambda |u_N|^2|\nabla u_N|^2-2\alpha\lambda (\nabla u_N)^2\overline{u}_N^2\Big{]}dxdt\\
=:&-2\alpha E[f(u_N)]dt+\mathcal{A}_1dt+\mathcal{A}_2dt+\mathcal{A}_3dt,
\end{align*}
where we have used the fact $\Delta(|u|^2u)=2|u|^2\Delta u+4u|\nabla u|^2+2\overline{u}(\nabla u)^2+u^2\Delta\overline{u}$ and $\mathcal{A}_1$, $\mathcal{A}_2$ and $\mathcal{A}_3$ are estimated as follows.
\begin{align*}
|\mathcal{A}_1|:=&\left|Re E\int_0^1\Big{[}4\lambda\bi(\Delta\overline{u}_N)u_N|\nabla u_N|^2
+2\lambda\bi(\Delta\overline{u}_N)\overline{u}_N(\nabla u_N)^2\Big{]}dx\right|\\
\le&\frac{\alpha}{16} E\|\Delta u_N\|_0^2+C E\left[\|u_N\|_{L^{\infty}}^2\|\nabla u_N\|_{L^4}^2\right]\\
\le&\frac{\alpha}{16} E\|\Delta u_N\|_0^2+C E\left[\|u_N\|_{L^{\infty}}^4+\|\Delta u_N\|_0\|\nabla u_N\|_{0}^3\right]\\
\le&\frac{\alpha}8 E\|\Delta u_N\|_0^2+C E\left[\|u_N\|_1^{4}+\|u_N\|_1^{6}\right]\\
\le&\frac{\alpha}8 E\|\Delta u_N\|_0^2+C,
\end{align*}
where we have used the uniform boundedness of $\|u_N\|_1^{2p}$ for $p\ge1$ in $\romannumeral2)$, the continuous embedding $H^1\hookrightarrow L^{\infty}$ for $\R^1$ and the interpolation of $L^4$ between $L^2$ and $H^1$. Similarly, based on the continuous embedding $H^1\hookrightarrow L^{6}$ and $H^1\hookrightarrow L^{8}$, we have
\begin{align*}
|\mathcal{A}_2|:=&\left|Re E\int_0^1\Big{[}\lambda^2\bi(\Delta\overline{u}_N)|u_N|^4
-4\alpha\lambda(\Delta\overline{u}_N)u_N|u_N|^2\Big{]}dx\right|\\
\le&\frac{\alpha}8 E\|\Delta u_N\|_0^2+C E[\|u_N\|_{L^8}^{8}+\|u_N\|_{L^6}^{6}]\\
\le&\frac{\alpha}8 E\|\Delta u_N\|_0^2+C
\end{align*}
and
\begin{align*}
|\mathcal{A}_3|:=&\left|Re E\int_0^1\Big{[}-4\alpha\lambda |u_N|^2|\nabla u_N|^2-2\alpha\lambda (\nabla u_N)^2\overline{u}_N^2\Big{]}dx\right|
\le C E\|u_N\|_1^4\le C.
\end{align*}
Thus, we obtain
\begin{align*}
 E[\mathcal{A}]\le-2\alpha E[f(u_N)]dt+\frac{\alpha}4 E\|\Delta u_N\|_0^2+C.
\end{align*}
The estimate of $\mathcal{C}$ is similar with that of $\mathcal{A}$, and we derive 
%\begin{align*}
$ E[\mathcal{C}]\le\frac{\alpha}4 E\|\Delta u_N\|_0^2+C.
$ %\end{align*}
 Taking expectation on both sides of \eqref{2norm} yields
\begin{align*}
d E f(u_N)+2\alpha E f(u_N)dt\le\frac{\alpha}2 E\|\Delta u_N\|_0^2dt+Cdt
\le\alpha E f(u_N)dt+Cdt.
\end{align*}
Multiplying both sides of above equation by $e^{\alpha t}$ and taking integral from $0$ to $t$, we conclude the uniform boundedness of $ E f(u_N(t))$
\begin{align*}
 E f(u_N(t))\le e^{-\alpha t} E f(u_N(0))+\frac{C}{\alpha}(1-e^{-\alpha t}),
\end{align*}
which yields the uniform boundedness of $ E\|\Delta u_N\|_0^2$ based on \eqref{5.6}. As the norm $\|u_N\|_2$ is equivalent to $\|\Delta u_N\|_0$ under Dirichlet boundary condition, we complete the proof.
\qed

\subsection{The proof of uniqueness of the solution for (\ref{BE})}\label{5.2}
Suppose that $U$ and $W$ are two solutions of the scheme, then it follows
$$U-W=\mathbf{i}\tau\Delta\big{(}U-W\big{)}+\mathbf{i}\lambda\frac{\tau}2\pi_N\Big{[}\big{(}|U|^2U-|W|^2W\big{)}+|e^{-\alpha\tau}u_N^{k-1}|^2(U-W)\Big{]}.$$
Multiply the equation above by $\overline{U}-\overline{W}$, integrate in space and take the real and imaginary part respectively, we have
\begin{align*}
&\|U-W\|^2_0\leq\frac{\tau}2\|f(U)-f(W)\|_{L^{\frac43}}\|U-W\|_{L^4},\\
&\|\nabla(U-W)\|^2_0\leq\frac12\|f(U)-f(W)\|_{L^{\frac43}}\|U-W\|_{L^4}+\frac{\lambda}2\|e^{-\alpha\tau}u_N^{k-1}\|_{L^4}^2\|U-W\|_{L^4}^2,
\end{align*}
where $f(U):=|U|^2U$ and
\begin{align*}
\|f(U)-f(W)\|_{L^{\frac43}}=&\left(\int_0^1\Big{|}|U|^2U-|W|^2W\Big{|}^{\frac43}dx\right)^{\frac34}\\
=&\left(\int_0^1\Big{|}|U|^2(U-W)+|W|^2(U-W)+UW(\overline{U}-\overline{W})\Big{|}^{\frac43}dx\right)^{\frac34}\\
\leq&\left(\int_0^1\Big{|}|U|^2+|W|^2+|UW|\Big{|}^2dx\right)^{\frac12}\left(\int_0^1|U-W|^4dx\right)^{\frac14}\\
\leq&\big{\|}|U|+|W|\big{\|}_{L^4}^2\|U-W\|_{L^4}.
\end{align*}
Since
\begin{align*}
\|U-W\|_{L^4}^4\leq&\|U-W\|_0^3\|\nabla(U-W)\|_0\\
\leq&\left(\frac{\tau}2\|f(U)-f(W)\|_{L^{\frac43}}\|U-W\|_{L^4}\right)^{\frac32}\bigg{(}\frac12\|f(U)-f(W)\|_{L^{\frac43}}\|U-W\|_{L^4}\\
&+\frac{|\lambda|}2\|e^{-\alpha\tau}u_N^{k-1}\|_{L^4}^2\|U-W\|_{L^4}^2\bigg{)}^{\frac12}\\
\leq&\frac14{\tau}^{\frac32}\big{\|}|U|+|W|\big{\|}_{L^4}^3\left(\big{\|}|U|+|W|\big{\|}_{L^4}^2+|\lambda|\|u_N^{k-1}\|_{L^4}^2\right)^{\frac12}\|U-W\|_{L^4}^4\\
\leq&\frac14{\tau}^{\frac32}\left(\big{\|}|U|+|W|\big{\|}_{L^4}^4+|\lambda|\big{\|}|U|+|W|\big{\|}_{L^4}^3\|u_N^{k-1}\|_{L^4}\right)\|U-W\|_{L^4}^4,
\end{align*}
if $U\not=W$, then 
\begin{align*}
1\leq&\frac14{\tau}^{\frac32}\left(\big{\|}|U|+|W|\big{\|}_{L^4}^4+|\lambda|\big{\|}|U|+|W|\big{\|}_{L^4}^3\|u_N^{k-1}\|_{L^4}\right)\\
\leq&C_0\tau^{\frac32}\left(\big{\|}|U|+|W|\big{\|}_{L^4}^4+|\lambda|\big{\|}|U|+|W|\big{\|}_{L^4}^6+|\lambda|\|u_N^{k-1}\|^2_{L^4}\right).
\end{align*}

For cases $\lambda=0$ or $-1$, the $L^4$-norm of the solutions are uniformly bounded. So $C_0\tau^{\frac32}>1$,
which do not hold when $\tau$ is sufficiently small. For case $\lambda=1$, according to the fact that
$$\big{\|}|U|+|W|\big{\|}_{L^4}^6\leq\big{\|}|U|+|W|\big{\|}_0^{\frac32}\big{\|}\nabla(|U|+|W|)\big{\|}_{0}^{\frac92}\leq N^{\frac92}\big{\|}|U|+|W|\big{\|}_0^6,$$
we have $C_0N^{\frac92}\tau^{\frac32}>1$, which is also a contradiction when $\tau$ is sufficiently small.

Thus, the numerical solution for (\ref{BE}) is unique.\qed

\section*{Acknowledgement}

We would like to thank Prof. Zhenxin Liu for useful discussions and comments on proofs of ergodicity of the schemes. Also, we are very grateful to Prof. Arnulf Jentzen and Prof. Xiaojie Wang for their helpful suggestions, which lead to many improvements in this article.

\nocite{*}  %%%%all the references in bib appear
\bibliography{wangxu}
\bibliographystyle{plain}

\end{document}